\documentclass[12pt, reqno]{amsart}

\usepackage{geometry}
\usepackage{amsmath}    % 公式编号默认在右
\usepackage{amssymb, amsthm, mathrsfs}

\usepackage{graphicx, float, subfigure, tikz}
\usepackage[all, cmtip]{xy}

\usepackage{helvet, courier, type1cm}
\usepackage[english]{babel}
\usepackage[T1]{fontenc}
\usepackage{xcolor, url, bookmark, upref, dsfont}
\newtheorem{theorem}{Theorem}[section]
\newtheorem*{theorem*}{Theorem}
\newtheorem{lemma}[theorem]{Lemma}

\newtheorem{proposition}[theorem]{Proposition}

\newtheorem*{comment*}{Comment}
\newtheorem*{definition*}{Definition}
\newtheorem*{remark*}{Remark}

\newtheorem*{observation*}{Observation}
\newtheorem{assumption}{Assumption}
\newtheorem*{assumption*}{Assumption}

\theoremstyle{definition}
\newtheorem{definition}[theorem]{Definition}

\theoremstyle{remark}
\newtheorem{remark}[theorem]{Remark}

\newcommand{\sgn}{\mathrm{sgn}}

\title[Gaussian limit for Pfaffian point processes]{Gaussian limit for Pfaffian point processes}

\author{Kai Wang}
\address{School of Mathematical Sciences, Fudan University, Shanghai, 200433, China.}
\email{kwang@fudan.edu.cn}

\author{Mei Xu}
\address{School of Mathematical Sciences, Fudan University, Shanghai, 200433, China.}
\email{meixu21@m.fudan.edu.cn}

\thanks{K.Wang is supported by grants NSFC (12231005,12326376) and the Shanghai Technology Innovation Project (21JC1400800).}

\begin{document}

\begin{abstract}
	We prove a central limit theorem for linear statistics of  a broad class of Pfaffian point processes. As an application, we derive Gaussian limits for scaled linear statistics of step functions  in the Pfaffian  $\mathrm{Sine_4}$ and $\mathrm{Sine}_1$ processes.
\end{abstract}
%deepseek写的
%We establish the central limit theorem for linear statistics over Pfaffian point processes by introducing the finite rank commutator property (FRCP), which generalizes Soshnikov’s framework for determinantal processes with Hermitian kernels. As an application, we prove that the number of points in the Sine₄-process under scaling converges to a Gaussian distribution with variance asymptotically proportional to log L.



\maketitle

% 正文内容从这里开始
\section{introduction}
%\subsection{Point process}%A configuration $\xi \in {\rm Conf}(X)$ is called simple if $\xi(\{x\}) \leqslant 1 $ for all $x\in X$. In the simple case, $\xi$ can be %identified with a discrete subset of $X$.
%The point process $\mathbb{P}$ on $X$ is called simple if $\mathbb{P}$-almost every configuration is simple. For more details, see, e.g. %\cite{}.%A configuration $\xi$ on $X$ is a collection of points of $X$ such that any compact set contains finite points of $\xi$.
%Let ${\rm Conf}(X)$ be equipped with its Borel $\sigma$-algebra. %Let $\#_{A_i}$ denote the number of points of $\mathbb{P}$ located in $A_i$.%where $\mathbb{E}_{\mathbb{P}}$ denotes expectation with respect to $\mathbb{P}$.\par

%\subsection{Pfaffian point process}For more details of this definition,
%see, e.g. \cite{}.
%Here ${\rm Pf}(R)$ is the Pfaffian of an antisymmetric matrix $R$.
%We refer the reader to \cite{Aslaksen2001, DysonFreemanJ.1970} for the definition and properties of $\mathrm{Pf}$.

%Bekerman, Leblé and Serfaty proved a CLT for the linear statistics of one-dimensional log-gases in a transportation method.
%这里的感觉相关性不太大，可能得删一点

A determinantal point process is a random collection of points whose probability distribution can be expressed in terms of determinants. Such processes arise naturally in the study of various mathematical objects, including eigenvalues of random matrices, zeros of random analytic functions, and spanning trees on networks. Due to their wide applicability, determinantal point processes have been extensively studied; we refer to [Soshnikov2000] for a comprehensive review.

However, in several important applications—such as eigenvalue distributions in the Gaussian orthogonal ensemble (GOE), Gaussian symplectic ensemble (GSE), and the real Ginibre ensemble, as well as in random combinatorial structures like tilings—the correlation functions are instead described by Pfaffians of antisymmetric matrices. These systems are governed by Pfaffian point processes, a natural counterpart to determinantal processes.

To provide a precise definition of Pfaffian point processes, we first recall some preliminary notation.

Let $X$ be a locally compact Polish space, and let
$\rm Conf(X)$ denote the space of all locally finite configurations on $X$,  i.e.,
\begin{align*}
	{\rm Conf}(X) = \big\{ \xi = \textstyle\sum\limits_{i} \delta_{x_i} ~|    x_i\in X ~ {\rm and}~ \xi(B) < + \infty ~{\rm for ~any ~compact}~ B \subset X \big\}.
\end{align*}
This configuration space is a Polish subspace  of $\mathcal{M}(X)$, the space of Radon measures on  $X$, equipped with the vague topology. A point process on $X$ is a Borel probability measure  $\mathbb{P}$ on ${\rm Conf}(X)$. For further background, see, e.g., \cite{Daley1988, Lenard1973, Lenard1975.1, Lenard1975}.

Now, let $\mu$ be a Radon measure on $X$.
A locally integrable function $\rho_k: X^k \rightarrow [0,+\infty)$  is called the $k$-th ($k\in\mathbb{N}_+$) correlation function of the point process $\mathbb{P}$ (with respect to the reference measure $\mu$)
if for any $m\in \mathbb{N}_+,$   disjoint measurable compact subsets $A_1,\cdots,A_m$ of $X$, and   positive integers $k_1,\cdots,k_m$ with $\sum\limits_{i=1}^m k_i = k$, the following identity holds:
\begin{align*}
	\mathbb{E}_{\mathbb{P}} \Big[ \prod\limits_{i = 1}^m  \#_{A_i} (\#_{A_i}-1)\cdots (\#_{A_i} - k_i +1) \Big]
	= \int_{A_1^{k_1}\times \cdots \times A_m^{k_m}} \rho_k(x_1,\cdots,x_k) d\mu(x_1)\cdots d\mu(x_k),
\end{align*}
where $\#_A (\xi)=\xi(A) $ counts the points of $\xi$ in $A$.   For additional details, we refer to \cite{Daley1988, Forrester2010, kallenberg1986random, Soshnikov2000}.

Let $K: X\times X \rightarrow M_2(\mathbb{C})$ be a matrix-valued kernel of the form
\begin{align*}
	K(x,y) =
	\left(
	\begin{array}{cc}
		K_{1,1}(x,y) &  K_{1,2}(x,y)\\
		K_{2,1}(x,y) &  K_{2,2}(x,y)
	\end{array}
	\right),
\end{align*}
satisfying the following symmetry conditions:
\begin{align*}
	K_{1,1}(x,y)=K_{2,2}(y,x),\,K_{1,2}(x,y) = -K_{1,2}(y,x), \quad \text{ and } K_{2,1}(x,y) = -K_{2,1}(y,x).
\end{align*}
Define the matrix \begin{align*}
	Z =\left(
	\begin{array}{cc}
		0 &  1\\
		-1 &  0,
	\end{array}
	\right).
\end{align*} and consider the antisymmetric kernel
\begin{align}\label{def-Pfa-ker}
	\mathbb{K}(x,y) = ZK(x,y).
\end{align}
Then, $\mathbb{K}(x,y)^T = -\mathbb{K}(y,x)$ for any $x,y\in X.$
A point process $\mathbb{P} $ is called a Pfaffian point process if its $k$-th  correlation function $\rho_k$
(for every $k\in{\mathbb N}_+$) admits the representation
\begin{align*}
	\rho_k(x_1,\cdots,x_k)  = \mathrm{Pf}\big[\mathbb{K}(x_i,x_j)\big]_{i,j = 1}^k.
\end{align*}
For further details, see \cite{Borodin_2005, bufetov2019number, Bufetov2021, forrester2011pfaffian, Kargin_2013, Koshida_2021, Matsumoto_2013, rains2000correlation, soshnikov2003janossy}.

A prominent example of Pfaffian point processes arises as the bulk scaling limits of Pfaffian Gaussian ensembles, known as the Pfaffian sine processes \cite{Forrester2010}.
Denote the functions
$$S(x):=\frac{sin(\pi x)}{\pi x},\, IS(x)=\int_{0}^{x} S(x) dx. $$
The orthogonal sine process (or Pfaffian  $\mathrm{Sine}_1$-process), is the
Pfaffian point process on $\mathbb R$ with the matrix  kernel
\begin{align*}
	\mathbb{K}_{\mathrm{Sine}_1}(x,y)= ZK_{\mathrm{Sine}_1}(x,y),\quad
	K_{\mathrm{Sine}_1}(x,y)
	=
	\left(\begin{array}{cc}
		S(x-y) & S^{\prime}(x-y)-\frac{1}{2}sgn(x-y)  \\
		IS(x-y) & S(x-y),
	\end{array}\right)
\end{align*}
where $ Z =\left(
\begin{array}{cc}
	0 &  1\\
	-1 &  0,
\end{array}
\right)$ as before.
Similarly, the symplectic sine process (or   Pfaffian $\mathrm{Sine}_4$ process), is the Pfaffian point process on $\mathbb R$ with
the matrix   kernel
\begin{align*}
	\mathbb{K}_{\mathrm{Sine}_4}(x,y)= ZK_{\mathrm{Sine}_4}(x,y),\quad
	K_{\mathrm{Sine}_4}(x,y)
	= \frac{1}{2}
	\left(\begin{array}{cc}
		S(x-y) & S^{\prime}(x-y)  \\
		IS(x-y) & S(x-y)
	\end{array}\right).
\end{align*}

In the paper we will study the asymptotic behavior of linear statistics
\begin{align*}
	S_f(\xi) = \int_X f d\xi = \sum\limits_{x\in \xi} f(x),\quad \xi \in {\rm Conf}(X),
\end{align*}
for test functions $f$  in a scaling limit. For the Pfaffian point process $\mathbb{P}$ with the matrix kernel $\mathbb{K}$ of the form  \eqref{def-Pfa-ker},  the expectation and variance of   $S_f$
\begin{align}\label{Math-Exp-For}
	\mathbb{E}_{\mathbb{P}} [S_f] = \int_X f(x) A(x,x) d\mu(x),
\end{align}
\begin{align}\label{Var-For}
	\mathrm{Var}_{\mathbb{P}} (S_f) = \int_X |f(x)|^2 A(x,x) d\mu(x) - \int_{X^2} f(x) \overline{f(y)}\,\mathrm{det} K(x,y) d\mu(x)d\mu(y).
\end{align}

%\subsection{Gaussian limit for Pfaffian point processes}
Central limit theorems (CLTs) for point processes have been established in many cases. For example, Leblé \cite{Leble2021} proved a central limit theorem for the fluctuations of linear statistics in the $\mathrm{Sine}_{\beta}$ process,  requiring test functions to be of class $C^4$ with compact support. His approach combined the DLR equations, Laplace transform techniques, and transportation methods. Subsequently,  Lambert \cite{Lambert2021} eveloped a simpler proof of a CLT for linear statistics of circular $\beta$-ensembles, , which holds at nearly microscopic scales for $C^3$ functions and consequently yields a central limit theorem for $\mathrm{Sine}_{\beta}$ processes. Related results can be found in \cite{BLS2018, Valkó2020}, among other works.

The methods for determinantal point processes differ significantly from those used. The foundation was laid by Costin and Lebowitz \cite{Costin1995}, who established a CLT for particle counts in the determinantal point process on   $\mathbb{R}$ with the kernel $\sin(\pi(x-y))/\pi(x-y).$ Soshnikov \cite{Soshnikov2000fluctuations,Soshnikov2002} later extended this work, developing a comprehensive central limit theorem for general determinantal point processes with Hermitian kernels through the moment method. Breuer and Duits \cite{Breuer&Duits2017} further applied the moment method to obtain CLTs for orthogonal polynomial ensembles and biorthogonal ensembles where the underlying biorthogonal family satisfies a recurrence relation. For Pfaffian point processes, progress has been more limited. To date, the most general result is due to Kargin \cite{Kargin_2013}, who proved a central limit theorem for Pfaffian point processes with finite-rank kernels.

In this work, we establish a central limit theorem  for a broad class of Pfaffian point processes. Our approach introduces a new tool called the finite-rank commutator property (FRCP), which quantifies how closely the kernel admits a tractable symmetry structure. Our main result extends Soshnikov's framework for  determinantal processes   to the Pfaffian setting under FRCP. Key examples include the Pfaffian $\mathrm{Sine}_4$  and   $\mathrm{Sine}_1$ processes, which satisfy this property.

To formalize FRCP, we first recall essential operator constructions. Let
$A,B,D,A^{\dag}$ denote the integral operator with  kernel functions $a(x,y), b(x,y), d(x,y), a(y,x)$.  Throughout the paper, when  $A$ is trace class, we assume   $a(x,y)$ is sufficiently smooth so that
$$\int_{X^k}
a(x_1,x_2) a(x_2,x_3)\cdots a(x_k,x_1) \,d\mu(x_1)\cdots d\mu(x_k) = \mathrm{Tr} (A ^k )=\mathrm{Tr} (A ^{{\dag}k} ),\quad k\in\mathbb{N}_+.$$
For $f,g\in L^2(X,\mu)$, we define the rank-1 operator  $f\otimes g$  by
\begin{align*}
	(f\otimes g)(h) =  \langle h,g\rangle_{L^2(X,\mu)} f,\quad \mathrm{for}\quad h\in L^2(X,\mu).
\end{align*}
Then the definition of FRCP is as follows.
\begin{definition} Let $\mathbb{P}$ be a Pfaffian point process with the matrix-valued kernel $$ \mathbb{K}(x,y)=ZK(x,y), K(x,y)= \lambda \left({\begin{array}{cc}a(x,y) &d(x,y)\\b(x,y)&a(y,x)\end{array}} \right)  (\lambda=\frac{1}{2}\, \mathrm{or}\, 1)$$ defined on a domain $X$ with respect to the measure $\mu$.  We say $\mathbb{P}$ has the finite-rank commutator property (FRCP) if there exist a constant
	$N$
	and functions
	$$f^{(i)},g^{(i)},h^{(i)},e^{(i)}\in L^2(X,\mu)(i=1,\cdots,N),$$
	such that
	\begin{align*}
		&A^{\dag} B -B  A = \sum\limits_{i=1}^{N}  f^{(i)}\otimes g^{(i)},\\
		&D B -(\alpha A^{2}+\beta A)  = \sum\limits_{i=1}^{N} h^{(i)}\otimes e^{(i)},
	\end{align*}
	where $\alpha,\beta\in\mathbb{C}$ satisfy $\alpha +\beta =\begin{cases}
		1, &  \mathrm{if}\, \lambda=\frac{1}{2};\\
		0, &  \mathrm{if}\,\lambda=1,
	\end{cases} $
	The tuple  $\{N,f^{(i)},g^{(i)},h^{(i)},e^{(i)},\alpha,\beta\}$ parametrizes the FRCP for $\mathbb{P}$.
\end{definition}
As demonstrated in Section~\ref{section-CLT-for-sine-beta}, both the   both the Paffian $\mathrm{Sine}_4$-process and the  Paffian $\mathrm{Sine}_1$-process have FRCP.

Building on this framework, we now present our central limit theorem for Pfaffian point processes with FRCP.

\begin{theorem}\label{Th2} Fixed an integer $N$.  Consider a family $\{\mathbb{P}_L\}_{L\geqslant 0}$ be a family of Pfaffian point processes having FRCP  characterized by the data  $\{N,f^{(i)}_{L},g^{(i)}_{L},h^{(i)}_{L},e^{(i)}_{L},\alpha_L,\beta_L\}$ with the matrix-valued kernel
	$$\mathbb{K}_L(x,y)=ZK_L(x,y),K_L(x,y)= \lambda_L \left({\begin{array}{cc}a_L(x,y) &d_L(x,y)\\b_L(x,y)&a_L(y,x)\end{array}} \right) (\lambda_L= \frac{1}{2} \,\,\text{or}\, \,1 ),$$
	defined on the domain $X_L$ with respect to the measure $\mu_L$. Under the following conditions:
	\begin{itemize}\label{assump_of_th3.1_2}\label{equ-in-assup4}
		\item[(i)] $\mathrm{Var}_{\mathbb{P}_L} (\#_{X_L}) \rightarrow +\infty$ as $L\rightarrow +\infty;$
		\item[(ii)]  Operators $A_L,B_L,D_L,A^{\dag}_L$ are bounded  with   $\sup\limits_L (\|A_L\|+\|A^{\dag}_L\|+|\alpha_L|+|\beta_L|)<+\infty;$
		\item[(iii)] $A_L$ is trace-class and satisfies $\|A_L- A_L^2\|_1 = o\big(\mathrm{Var}_{\mathbb{P}_L}(\#_{X_L})\big)^{\delta} $ for some $\delta>0;$
		\item[(iv)]  for all $i,j = 1,\cdots,N ,$ and $m,n\in \mathbb{Z}_{\geqslant 0},$ the inner products
		$$
		\langle D_L A_L^{{\dag}\, m} f^{(i)}_{L},  {A_L}^{ n} g^{(j)}_{L}\rangle,
		\langle D_L A_L^{{\dag}\, m} f^{(i)}_{L},  {A_L}^{ n} e^{(j)}_{L}\rangle,
		\langle h^{(i)}_{L}, {A_t}^{ n} g^{(j)}_{L}\rangle,
		\langle h^{(i)}_{L},  {A_t}^{ n} e^{(j)}_{L}\rangle=  o\big(\mathrm{Var}_{\mathbb{P}_L}(\#_{X_L})\big),
		$$
	\end{itemize}	
	then the normalized counting statistic
	\begin{align*}
		\frac{\#_{X_L}- \mathbb{E}_{\mathbb{P}_L}[\#_{X_L}] }{\sqrt{\mathrm{Var}_{\mathbb{P}_L}(\#_{X_L})}}
	\end{align*}				
	converges in distribution to   $N(0,1)$ as $L \rightarrow +\infty$.
\end{theorem}
We apply Theorem~\ref{Th2} to establish CLTs for both  the Pfaffian  $\mathrm{Sine}_1$-process and the  Pfaffian $\mathrm{Sine}_4$-process. Furthermore, by a similar argument we extend these results to step function statistics .

\begin{theorem}\label{Th-sine-beta-step-func} Fixed $\beta =1, 4$. Consider a step function  $\varphi=\sum\limits_{i = 1}^N\lambda_i\chi_{(a_i,b_i)},$
	where $\lambda_i\in\mathbb{R}\setminus \{0\}$, $(a_i,b_i)$ are disjoint intervals. Define the scaled version  $\varphi_L(x) = \varphi(\frac{x}{L})$.
	Then the normalized statistic
	\begin{align*}
		\frac{S_{\varphi_L} - \mathbb{E}_{\mathrm{Sine}_\beta} [S_{\varphi_L}] }{\sqrt{\mathrm{Var}_{\mathrm{Sine}_\beta }  (S_{\varphi_L})}}
	\end{align*}				
	converges in distribution to  $N(0,1)$ as $L \rightarrow +\infty$.
\end{theorem}

The remainder of this paper is structured as follows: Section~\ref{section-pf-main-results} proves the main theorem using moment analysis techniques, while Section~\ref{section-CLT-for-sine-beta} establishes the finite-rank commutator property for both the Pfaffian  $\mathrm{Sine}_4$  and  $\mathrm{Sine}_1$ processes and subsequently extends the central limit theorem to step functions by means of trace class operator estimates.

\section{proof of theorem \ref{Th2}}\label{section-pf-main-results}
%In this section, we will use a classic result about convergence in distribution to prove Theorem~\ref{Th2}. For this purpose, it is necessary to establish the relationship between the cumulant and variance of the linear statistic. This relationship is obtained thanks to the cumulant expression of linear statistics in the Pfaffian point process that we provide.

%\subsection{Proof of Theorem~\ref{Th2}}\label{proof of theorem}
We will prove Theorem \ref{Th2} by the moment method. Recall that for a real-valued random variable $\eta$ with all finite moments, the cumulants $c_n(\eta),$ $n \in \mathbb{N}_+,$ are defined through the Taylor coefficients of the logarithm of  characteristic function:
\begin{align}\label{cumulants}
	\log \mathbb{E} \big[\exp (iz\eta )\big] =
         \sum\limits_{n=1}^{\infty} c_n(\eta) \frac{(iz)^n}{n!}.
\end{align}
From \cite[Section 2.4]{Lu}, the $N$-th moments of $\eta$ can be expressed with cumulants as
	\begin{align}\label{momentsbycumulants}
		\mathbb{E}[\eta^N] = \sum_{\substack{
				\mathrm{over}~\mathrm{partitions}\\
				\mathcal{P}=\{P_1,\cdots,P_q\}~\mathrm{of}~[N]
			}
		} c_{|P_1|} (\eta) \,c_{|P_2|}(\eta) \,\cdots \,c_{|P_q|}(\eta).
\end{align}
Here and throughout the paper, we will denote  $|\cdot|$ of a set as the number of its elements,  and define $[N] = \{1,\cdots,N\} $.
We will use the following Lemma~\ref{CLT-and-cn}, a standard result from the convergence of cumulants to the convergence in distribution for the standard normal law $N(0,1)$, see, e.g. \cite[Theorem~3.3.26]{Du}, \cite[Corollary to Theorem~7.3.3]{Lu} and \cite[Lemma~3]{Soshnikov2002}.
%We will use the following lemma.
\begin{lemma}\label{CLT-and-cn}
	Let $\{\eta_L\}_{L\geqslant0}$ be a family of real-valued random variables such that $c_1(\eta_L) = 0,$ $c_2(\eta_L) = 1$ for all $L>0,$ and $c_n(\eta_L) \rightarrow 0$ as $L \rightarrow +\infty$ when $n$ is sufficiently large.
	Then    $\eta_L$ converges in distribution to $N(0,1)$ as $L \rightarrow +\infty$.
\end{lemma}
According to Lemma~\ref{CLT-and-cn}, we are going to show that the $n$-th cumulant of the normalized linear statistics
\begin{align*}
	\frac{\#_{X_L}- \mathbb{E}_{\mathbb{P}_L}[\#_{X_L}] }{\sqrt{\mathrm{Var}_{\mathbb{P}_L}(\#_{X_L})}}
\end{align*}							
converges to zero as $L\rightarrow \infty$ when $n$ is sufficiently large.
It suffices to show that, there exists $n_0\in\mathbb{N}_+$ such that
\begin{align}\label{estimate-cn-TH2}
	\mathrm{for}\quad n\geqslant n_0,\quad c_n(\#_{X_L})
	=o\big(\mathrm{Var}_{\mathbb{P}_L}(\#_{X_L})\big)^{\frac{n}{2}}\quad\mathrm{as}\quad L\rightarrow +\infty.
\end{align}
We use the following proposition to give an expression for the cumulant of the linear statistics in the case of Pfaffian point process.
%We delay the proof of Proposition~\ref{lemma-cumulants-Pfaffian} to Section~\ref{subsubsection-pf-of-cumu-Pfaffian}.
By the way, the corresponding result in the case of determinantal is established by Soshnikov in \cite{Soshnikov2000CLT}.
\begin{proposition}\label{lemma-cumulants-Pfaffian}
	Let $\mathbb{P}$ be a Pfaffian point process on a locally compact Polish space $X$ with the matrix-valued kernel function $\mathbb{K}(x,y)=ZK(x,y)$ with respect to the reference Radon measure $\mu$. Then for any bounded measurable compactly supported function $f:X\to\mathbb{R}$, the $n$-th $(n\in \mathbb{N}_+)$ cumulant of the linear statistic $S_f$ can be expressed by
	%\mathbb{K}(x,y)=\big[\mathbb{K}_{i,j}(x,y)\big]_{i,j=1}^2
	\begin{align*}
		c_n(S_f)
		&=\frac{1}{2} \sum_{\substack{(l_1,\cdots,l_k)\in\mathbb{N}_+^k,\\
				l_1+\cdots+l_k = n }} (-1)^{k-1} \frac{n!}{l_1!\cdots l_k!}\cdot\frac{1}{k} \int_{X^k} f^{l_1}(x_1)f^{l_2}(x_2)\cdots f^{l_k}(x_k)\\
		&\qquad\qquad\qquad\cdot\mathrm{Tr}\big(K(x_1,x_2)K(x_2,x_3)\cdots K(x_k,x_1)\big)d\mu(x_1)\cdots d\mu(x_{k}),
	\end{align*}
	here we default $m+1=1$.
\end{proposition}

\begin{proof}
	We begin the proof with an argument for the $n$-th  correlation functions of $\mathbb{P}$.
	Write
$
		 {K}(x,y) = \left(
		\begin{array}{cc}
			 {K}_{1,1}(x,y) & {K}_{1,2}(x,y)\\
			 {K}_{2,1}(x,y) &   {K}_{2,2}(x,y)
		\end{array}
		\right)
$
with  $K_{1,2}(x,y) = -K_{1,2}(y,x), K_{2,1}(x,y) = -K_{2,1}(y,x), K_{2,2}(x,y)=K_{1,1}(y,x).$
Clearly $$\mathbb{K}(x,y) = \left(
		\begin{array}{cc}
			  {\mathbb K}_{1,1}(x,y) & {\mathbb K}_{1,2}(x,y)\\
			 {\mathbb K}_{2,1}(x,y) &   {\mathbb K}_{2,2}(x,y)
		\end{array}
		\right)=ZK(x,y)=\left(
		\begin{array}{cc}
			 {K}_{2,1}(x,y) & {K}_{2,2}(x,y)\\
			 -{K}_{1,1}(x,y) &   -{K}_{1,2}(x,y)
		\end{array}
		\right)$$
satisfying  $\mathbb{K}(x,y)^T = -\mathbb{K}(y,x).$
	%Recall that by \cite{DysonFreemanJ.1970}, if $R$ is a $2n$-by-$2n$ antisymmetric matrix, the Pfaffian of $R$ is defined as follows:%这里之前已经出现过这个定义了，暂时放着，是否删掉
	%\begin{align*}
	%	\mathrm{Pf}(R) =.
	%\end{align*}= \mathrm{Pf}\big[\mathbb{K}(x_i,x_j)\big]_{i,j=1}^{n}
For the fixed integer $n$ and  $(x_1,\cdots,x_n)\in X^n$, 	denote $R   = \big[\mathbb{K}(x_i,x_j)\big]_{i,j=1}^n.$ Then by definition, the $n$-th correlation function of $\mathbb{P}$ is
\begin{equation}\label{(7)}\rho_n (x_1, \cdots,x_n) = \mathrm{Pf}\big(R \big)= \frac{1}{2^n n!}\sum\limits_{\sigma\in S_{2n}} \sgn(\sigma) \prod\limits_{i=1}^n R_{\sigma(2i-1),\sigma(2i)}.\end{equation}
We claim that
\begin{equation}\label{(6)}{
		\rho_n (x_1, \cdots ,x_n)
		 =    \sum\limits_{\tiny{\substack{\mathrm{over~partitions}\\
				 \{K_1,\cdots,K_q\}~\mathrm{of}~[n]
			}
		}} \tiny{\prod\limits_{\alpha=1}^q }
		   \sum\limits_{\tiny{\substack{\mathrm{over~necklace ~arrangement}\\
				(\tau_{1},\cdots, \tau_{ |K_\alpha|})  ~\mathrm{of}~ K_\alpha
			}
		}}}{  {  { {\frac{(-1)^{|K_\alpha|-1}}{2}\mathrm{Tr}  (K(x_{\tau_{1}},x_{\tau_{2} })\cdots K(x_{\tau_{ |K_\alpha|}},x_{\tau_1}) )}.
}}
}\end{equation}
Here and thought the paper,  we regard a necklace arrangement for a number set $K$ as  a  permutation  of $K$ beginning with  the minimum number.

To prove the claim, we denote  $[[K]]=\bigcup\limits_{v\in K}\{2v-1,2v\}$ for a number set $K$. We call a permutation $\sigma\in S_{2n}$
as a Pfaffian cycle with respect to $(K, \tau)$, if $K=\{v_1,\cdots,v_k\}$ (here ordered by size from smallest to largest) is a subset of $[n]$ and $\tau=(\tau_1,\cdots,\tau_k)$ is a  necklace arrangement of $K$  such that  $\sigma$ is a permutation of the subset $[[K]]\subset [2n]$,  and leaves all elements outside
$[[K]]$ unchanged; and
for any $1\leqslant j\leqslant k $,    $\{\sigma_{2v_{j} },\sigma_{2v_{j}+1}\}=[[\tau_{j+1}]]=\{2\tau_{j+1}-1, 2\tau_{ j+1}\}$ ( here we default $\tau_{k+1} =\tau_{1}$).
Let $T(\tau)$ be the subset of  the Pfaffian cycle with respect to $(K, \tau)$. It is easy to see that $|T(\tau)|=2^{k}$.

For any permutation $\sigma \in S_{2n}$,  it is direct by induction to see that $\sigma$ admits a Pfaffian cycle decomposition.  That is,  there exist
a unique partitions
				$\mathcal{K} = \{K_1,\cdots,K_q\}$ of $[n]$,   necklace arrangements $\tau^{\alpha}$ of $K^\alpha$ and
$\sigma^{\alpha}\in T(\tau^\alpha)$ for $1\leqslant \alpha \leqslant q$,  and   permutations $\pi\in S_{2n},\rho\in S_n $ with
$$\{\pi(2i-1),\pi(2i)\} =\{2\rho(i)-1, 2\rho(i)\}\,\, \mathrm{for}~ i=1,\cdots,n;$$ $$ \pi(2i-1)=2\rho(i)-1,\pi(2i)=2\rho(i)~\mathrm{for}~ i=\tau^1_1,\cdots,\tau^q_1 ,$$
 such that
$$\sigma\circ \pi=\sigma^{1}\cdots\sigma^{q}. $$
A direct computation shows that
\begin{eqnarray*}
  sgn(\sigma)\prod\limits_{i=1}^n R_{\sigma(2i-1),\sigma(2i)} &=& sgn(\sigma\circ \pi)\prod\limits_{i=1}^n R_{(\sigma\circ \pi)(2i-1),(\sigma\circ \pi)(2i)} \\
   &=& \prod\limits_{\alpha=1}^q \sgn(\sigma^{\alpha})
   \prod_{v\in K_\alpha} R_{\sigma^{\alpha}_{2v-1},\sigma^{ \alpha}_{2v}}.
\end{eqnarray*}
Therefore,
\begin{equation}\label{(10)}\frac{1}{2^n n!}\sum\limits_{\sigma\in S_{2n}} \sgn(\sigma) \prod\limits_{i=1}^n R_{\sigma(2i-1),\sigma(2i)}= \sum\limits_{\tiny{\substack{\mathrm{over~partitions}\\
				 \{K_1,\cdots,K_q\}~\mathrm{of}~[n]
			}
		}} \prod\limits_{\alpha=1}^q
\sum\limits_{\tiny{\substack{\mathrm{over~\sigma^{\alpha}\in T(\tau^\alpha),}
				\mathrm{necklace} \\
		 \mathrm{arrangement}~ \tau^{\alpha} ~\mathrm{of}~  K_{\alpha}
			 	}
		}} \frac{\sgn(\sigma^{\alpha}) } {2}
 \prod_{v\in K_\alpha} R_{\sigma^{\alpha}_{2v-1},\sigma^{ \alpha}_{2v}}.
\end{equation}
It remains to compute the summation  with respect to  $T(\tau) $ for  a fixed necklace arrangement $\tau$ of a subset $K$. After composing with a permutation in $S_{2n}$, we may suppose that
$K=[k] $ and $\tau=(1,2,\cdots,k)$ with $k=|K|$ without loss of generalization. This implies that a permutation $\rho\in T(\tau)$ is equal to that $\{\rho_{2j},\rho_{2j+1}\}=\{2j+1,2j+2\}$ for $1\leqslant j\leqslant k$; here we default $\rho_{2k+1}=\rho_1$. Set
$$J=\{j: 1\leqslant j\leqslant n, \sigma_{2j}>\sigma_{2j+1}\}.$$ %Set $i_l=2$ when $k\in J$ and  $i_k=1$ when $l\notin J$ .
Then we have that $\sgn(\rho)=(2k-1)+|J|$ and
\begin{eqnarray*}
  \prod_{i \in K} R_{\rho_{2i-1},\rho_{2i}} &=&   \prod_{i=1}^k {\mathbb K}_{2-\chi_J(i-1), 1+\chi_J(i)} (x_i,x_{i+1}) \\
  &=& (-1)^{k-|J|} \prod_{i=1}^k {  K}_{1+\chi_J(i-1), 1+\chi_J(i)} (x_i,x_{i+1})  \\
  &=& (-1)^{k-|J|} \prod_{i=1}^k {  K}_{1+\chi_J(i-1), 1+\chi_J(i)} (x_i,x_{i+1});
\end{eqnarray*}
here $\chi_J$ is  an indicator function for the subset $J$ defined on $[k]$, and we default $\chi_J(0)=\chi_J(n)$.
Therefore,
\begin{eqnarray*}
 \sum\limits_{\rho\in T(\tau)}\sgn(\rho) \prod_{i \in K} R_{\rho_{2i-1},\rho_{2i}} &=&  \sum\limits_{J\in [k]} (-1)^{k-1}\prod_{i=1}^k {  K}_{1+\chi_J(i-1), 1+\chi_J(i)} (x_i,x_{i+1})\\
    &=&  (-1)^{k-1}\sum_{i_1,\cdots,i_k \in\{1,2\}} {K}_{i_1, i_2} (x_i,x_{i+1}) \\
      &=&  (-1)^{k-1}  \mathrm{Tr}\big (K(x_{ {1}},x_{ {2} })\cdots K(x_{k} ,x_{1})  \big).
\end{eqnarray*}
 Combing the above equation with  \eqref{(7)}  and  \eqref{(10)}, we obtain \eqref{(6)}, which finished the proof of the claim.

	Now let $  S_f(\xi) = \sum\limits_{\theta \in \xi} f(\theta)$ for any $\xi \in \mathrm{Conf}(X)$.  It is well known that
\begin{eqnarray}\label{(8)}
  \mathbb{E}_{\mathbb{P}}[S_f^N]  &=&  \mathbb{E}_{\mathbb{P}} \Big[\sum\limits_{\theta_1 \in \xi}f(\theta_1)  \sum\limits_{\theta_2 \in \xi}f(\theta_2)
		\cdots \sum\limits_{\theta_N \in \xi}f(\theta_N) \Big] \nonumber\\
     &=& \sum_{\substack{
				\mathrm{over}~\mathrm{patitions}\\
				 \{M_1,\cdots,M_r\} ~\mathrm{of}~[N]
			}
		}\mathbb{E}_{\mathbb{P}}[
		\sum\limits_{\theta_1 \neq \theta_2 \neq \cdots \neq \theta_r}
		f^{|M_1|}(\theta_1) f^{|M_2|}(\theta_2) \cdots f^{|M_r|}(\theta_r) ]\\
    &=& \sum_{\substack{
				\mathrm{over}~\mathrm{patitions}\\
				 \{M_1,\cdots,M_r\} ~\mathrm{of}~[N]
			}
		} \int_{X^r} f^{|M_1|}(x_1) f^{|M_2|}(x_2) \cdots f^{|M_r|}(x_r) \rho_r (x_1, \cdots ,x_r)d\mu(x_1)\cdots d\mu(x_r). \nonumber
\end{eqnarray}

	Substituting \eqref{(6)} into \eqref{(8)}, we obtain that the $N$-th moment of $S_f$ has the following form:
	\begin{align*}
		\mathbb{E}_{\mathbb{P}}[S_f^N] = \sum_{\substack{\mathrm{over~partitions}\\
				\mathcal{M} = \{M_1,\cdots ,M_r\} ~\mathrm{of}~ [N] }}
		\quad
		\sum_{\substack{\mathrm{over~partitions}\\
				\mathcal{K} = \{ K_1,\cdots,K_q\} ~\mathrm{of} ~[r]}}
		S_{\mathcal{M},\mathcal{K}},
	\end{align*}
where
\begin{align*}
S_{\mathcal{M},\mathcal{K}} =   \int_{X^r} f^{|M_1|}(x_1) f^{|M_2|}(x_2) \cdots f^{|M_r|}(x_r) {\prod\limits_{i=1}^q }
		   \sum\limits_{\tiny{\substack{\mathrm{over~necklace arrangement}\\
				 \tau^i   ~\mathrm{of}~ K_i
			}
		}}\frac{(-1)^{|K_i|-1}}{2}\\
 { {  { {\mathrm{Tr}  (K(x_{\tau^i_{1}},x_{\tau^i_{2} })\cdots K(x_{\tau^i_{ |K_\alpha|}},x_{\tau^i_1}) )}
}}
}
d\mu(x_1)\cdots d\mu(x_r) \\
  = {\prod\limits_{i=1}^q }
		   \sum\limits_{\tiny{\substack{\mathrm{over~necklace ~arrangement} \\
				 \tau^i   ~\mathrm{of}~ K_i
			}
		}}\frac{(-1)^{|K_i|-1}}{2} \int_{X^{|K_i|}} f^{|M_{\tau^i_1}|}(x_{\tau^i_1}) f^{|M_{\tau^i_2}|}(x_{\tau^i_2}) \cdots f^{|M_{\tau^i_{|K_i|}}|}(x_{\tau^i_{|K_i|}})\\
{ {  { {\mathrm{Tr}  (K(x_{\tau^i_{1}},x_{\tau^i_{2} })\cdots K(x_{\tau^i_{ |K_i|}},x_{\tau^i_1}) )}
}}
} d\mu(x_{\tau^i_1})\cdots d\mu(x_{\tau^i_{ |K_i|}})\\
 = {\prod\limits_{i=1}^q }
		   \sum\limits_{\tiny{\substack{\mathrm{over~necklace~ arrangement} \\
				 \tau^i   ~\mathrm{of}~ K_i
			}
		}}\frac{(-1)^{|K_i|-1}}{2} \int_{X^{|K_i|}} f^{|M_{\tau^i_1}|}(x_{1}) f^{|M_{\tau^i_2}|}(x_{2}) \cdots f^{|M_{\tau^i_{|K_i|}}|}(x_{|K_i|})\\
{ {  { {\mathrm{Tr}  (K(x_{ {1}},x_{ {2} })\cdots K(x_{ { |K_i|}},x_{ 1}) )}
}}
} d\mu(x_{ 1})\cdots d\mu(x_{ { |K_i|}})\\
\end{align*}
	To interchange the order of summation, we construct a new partition $\mathcal{P}$ of $[N]$ as follows:
	\begin{align*}
		\mathcal{P} = \{P_1,\cdots, P_q\},\quad\mathrm{where} \quad P_i =  \bigsqcup\limits_{j\in K_i} M_j,\quad\mathrm{for}\quad  i = 1,\cdots,q.
	\end{align*}
	Then $\{M_j\}_{j\in K_i}$ gives a partition of $P_i$ denoted by $\mathcal{P}_i$:
	\begin{align*}
		\mathcal{P}_i = \{P_{i,1},\cdots, P_{i,t_i}\} = \{M_j,~j\in K_i\}, \quad\mathrm{where}\quad t_i = |K_i|, \quad \mathrm{for}\quad i = 1, \cdots, q.
	\end{align*}
	It follows that
	\begin{align}\label{(3)}
		\mathbb{E}_{\mathbb{P}}[S_f^N]
		= \sum_{\substack{\mathrm{over~partitions}\\
				\mathcal{P} = \{P_1,\cdots,P_q\}~\mathrm{of}~[N]}}
		\quad
		\sum_{ \substack{ \mathrm{over}~ i = 1,\cdots, q,\\
				\mathrm{patitions} ~ \mathcal{P}_i ~ \mathrm{of} ~P_i: \\
				\mathcal{P}_i = \{P_{i,1},\cdots, P_{i,t_i}\} } }
	\prod\limits_{i=1}^q
   \sum\limits_{ {\substack{\mathrm{over~necklace~arrangement}\\
				 \tau^i    ~\mathrm{of}~  [t_i]
			}
		}}
		\frac{(-1)^{t_i-1}}{2}\Big( \int_{X^{t_i}} f^{|P_{i,\tau^i_1}|}(x_1)\nonumber \\ f^{|P_{i,\tau^i_2}|}(x_2)\cdots f^{|P_{i,\tau^i_{t_i}}|}(x_{t_i})
		 		 {  {  { {\mathrm{Tr}  (K(x_{ {1}},x_{ {2} })\cdots K(x_{ { t_i}},x_{ 1}) )}
}}
}  d\mu(x_1)\cdots d\mu(x_{t_i}) \Big)  \nonumber \\
		= \sum_{\substack{\mathrm{over~partitions}\\
				\mathcal{P} = \{P_1,\cdots,P_q\}~\mathrm{of}~[N]}}
		\prod\limits_{i=1}^q
		\Big( \sum_{ \substack{ \mathrm{over~partitions}~\mathcal{P}_i \\
				\mathrm{of}~ P_i: ~\mathcal{P}_i = \{P_{i,1},\cdots, P_{i,t_i}\} } }
	\frac{(-1)^{t_i-1}(t_i-1)!}{2}	\int_{X^{t_i}} f^{|P_{i,1}|}(x_1) f^{|P_{i,2}|}(x_2)\nonumber \\ \cdots
		f^{|P_{i,t_i}|}(x_{t_i})
		{\mathrm{Tr}  (K(x_{ {1}},x_{ {2} })\cdots K(x_{ { t_i}},x_{ 1}) )}~ d\mu(x_1)\cdots d\mu(x_{t_i})  \Big).
	\end{align}
	
	Comparing $\eqref{(3)}$ with $\eqref{momentsbycumulants}$ we arrive at
	%\begin{align*}
		%c_n(S_f) &=  \sum_{\substack{\mathrm{over~partitions}~\\
		%		\mathcal{P}=\{P_1,\cdots,P_k\}~\mathrm{of}~[n]}
		%}
		%\int_{X^k} f^{|P_1|}(x_1)f^{|P_2|}(x_2)\cdots
		%f^{|P_k|}(x_k)
		%\frac{1}{2^k} \sum\limits_{\sigma \in T(k)}
		%\mathcal{R}(\sigma)d\mu(x_1)\cdots d\mu(x_k).
	%\end{align*}
	%Using the expression of $\mathcal{R}(\sigma),$ the above formula completes the proof of Lemma \ref{c_n}.
%	Applying the expression of $\mathcal{R}(\sigma)$ to the formula, we have
	\begin{align*}
		c_n(S_f) &= \sum_{\substack{\mathrm{over~partitions}~\\
				\mathcal{P}=\{P_1,\cdots,P_k\}~\mathrm{of}~[n]}
		}
		\frac{(-1)^{k-1} (k-1)!}{2} \int_{X^k} f^{|P_1|}(x_1)f^{|P_2|}(x_2)\cdots
		f^{|P_k|}(x_k)
		\\
		&\qquad\qquad\qquad\cdot\mathrm{Tr}\big(K(x_1,x_2)K(x_2,x_3)\cdots K(x_k,x_1)\big)d\mu(x_1)\cdots d\mu(x_{k})  \\
		&= \sum_{\substack{(l_1,\cdots,l_k)\in\mathbb{N}_+^k,\\
				l_1+\cdots+l_k = n }} \frac{(-1)^{k-1} (k-1)!}{2}  \frac{n!}{l_1!\cdots l_k!}\cdot\frac{1}{k!} \int_{X^k} f^{l_1}(x_1)f^{l_2}(x_2)\cdots f^{l_k}(x_k) \\
		&\qquad\qquad\qquad\cdot\mathrm{Tr}\big(K(x_1,x_2)K(x_2,x_3)\cdots K(x_k,x_1)\big)d\mu(x_1)\cdots d\mu(x_{k}) \\
		&= \frac{1}{2} \sum_{\substack{(l_1,\cdots,l_k)\in\mathbb{N}_+^k,\\
				l_1+\cdots+l_k = n }} (-1)^{k-1} \frac{n!}{l_1!\cdots l_k!}\cdot\frac{1}{k} \int_{X^k} f^{l_1}(x_1)f^{l_2}(x_2)\cdots f^{l_k}(x_k)\\
		&\qquad\qquad\qquad\cdot\mathrm{Tr}\big(K(x_1,x_2)K(x_2,x_3)\cdots K(x_k,x_1)\big)d\mu(x_1)\cdots d\mu(x_{k}).
	\end{align*}
	The proof of Proposition~\ref{lemma-cumulants-Pfaffian} is finished by this equation.
\end{proof}

Applying  Proposition~\ref{lemma-cumulants-Pfaffian} in the case that $f=\chi_X$,  we have that for $n\in\mathbb{N}_+,$
\begin{align*}
	c_n(\#_{X })&=\frac{1}{2}
	\sum_{\substack{(l_1,\cdots,l_k)\in\mathbb{N}_+^k,\\
			l_1+\cdots+l_k = n }}  \frac{(-1)^{k-1}n!}{l_1!\cdots l_k!}\cdot\frac{1}{k}
	  \int_{X ^k} \mathrm{Tr}\big(K (x_1,x_2)K (x_2,x_3)\cdots K (x_k,x_1)\big)\\
		&\qquad\qquad\qquad\qquad\qquad\qquad\qquad\qquad\qquad\qquad\qquad\qquad d\mu (x_1)\cdots d\mu (x_{k})%\\
		%&=\frac{1}{2} \sum_{\substack{(l_1,\cdots,l_k)\in\mathbb{N}_+^k,\\
		%		l_1+\cdots+l_k = n }}  \frac{(-1)^{k-1}n!}{l_1!\cdots l_k!}\cdot\frac{1}{k}
		%\sum_{\substack{
		%	i_2,\cdots,i_k\in\{1,2\}
		%}}
	%	\int_{I ^k} K_{  i_1,i_2}(x_1,x_2)K_{ i_2,i_3}(x_2,x_3)\cdots K_{ i_k, i_1}(x_k,x_1)\\
	%	&\qquad\qquad\qquad\qquad\qquad\qquad\qquad\qquad\qquad\qquad\qquad\qquad d\mu_L(x_1)\cdots d\mu_L(x_k).		
\end{align*}
For   $k\in\mathbb{N}_+,$ denote
\begin{align*}
	V_k(\#_{X }) =\frac{1}{2}  \int_{X ^k} \mathrm{Tr}\big(K (x_1,x_2)K (x_2,x_3)\cdots K (x_k,x_1)\big) d\mu (x_1)\cdots d\mu (x_k).
\end{align*}
For $k,n\in\mathbb{N}_+,$ $k\leqslant n,$ define
\begin{align*}
	v(n,k)=\sum_{\substack{(l_1,\cdots,l_k)\in\mathbb{N}_+^k,\\
			l_1+\cdots+l_k = n }}  \frac{n!}{l_1!\cdots l_k!}\cdot\frac{1}{k!}.
\end{align*}
Then
\begin{align*}
	v(n,k)=v(n-1,k-1) + kv(n-1,k),
\end{align*}
where $v(0,k)=v(n,0)=0$ for all $k,n.$
Hence we have that for $n\geqslant 2$,
\begin{align}\label{eq-cn-and-Vk}
	c_n(\#_{X })
	&=\sum\limits_{k=1}^n(-1)^{k-1}(k-1)!v(n,k)V_k(\#_{X }) \nonumber \\
	&=\sum\limits_{k=1}^n(-1)^{k-1}\big((k-1)!v(n-1,k-1)V_k(\#_{X })  + k!v(n-1,k)V_{k}(\#_{X }) \big)\nonumber\\
	&=\sum\limits_{k=2}^{n}(-1)^{k-1}(k-1)! v(n-1,k-1)(V_{k}(\#_{X }) -V_{k-1}(\#_{X }) ).
\end{align}
In order to estimate $V_k $, we need the following lemmas.

\begin{lemma}\label{identity} Suppose that  the $2\times 2$ matrix-valued kernel function
${\mathbb K}(x, y) = ZK(x, y)$ is antisymmetric over a domain $X $. If
$$W_n=\int_{X^n} K(x_1,x_2)  K(x_2,x_3)  \cdots K(x_n,x_1) d\mu(x_1)\cdots d\mu(x_n)$$
is integrable, then $W_n$ is  a multiple of  identity matrix.
\end{lemma}
\begin{proof}  Since $\mathbb{K} (x,y)= ZK(x,y)$ is antisymmetric, we have that
\begin{align*}
	K(x,y)^T = ZK(y,x)Z^{-1}.
\end{align*}%这里把转置写成T
This follows  that
\begin{align*}
	W_n^T &= \int_{X^n} K(x_n,x_1)^T K(x_{n-1},x_n)^T \cdots K(x_1,x_2)^T  d\mu(x_1)\cdots d\mu(x_n)\\
	&= \int_{X^n} ZK(x_1,x_n)K(x_n,x_{n-1})\cdots K(x_2,x_1)Z^{-1} d\mu(x_1)\cdots d\mu(x_n)\\
	&=ZW_n Z^{-1}.
\end{align*}
It follows by an elementary matrix argument that $W_n $ is a multiple of   the 2-by-2 identity matrix.
\end{proof}

\begin{lemma}\label{claim-of-tr-inconsistent}
	%{\bf{Add condition, and remove $t$;} The equation \eqref{terms-of-T_k} takes the following form:}
	%By Assumption \ref{assump_of_th3.1_2}, \eqref{terms-of-T_k} is of the following form:
	%the sum of $\mathrm{Tr}\big(A_t^k\big)$ and traces of linear combinations of some rank-1 operators, which are products of operators with the following forms:
Let $\mathbb{P}$ be a Pfaffian point process with the matrix-valued kernel function $\mathbb{K}(x,y)=ZK(x,y)$ over a domain $X$ with respect to the reference measure $\mu$ and  have FRCP with $\{N,f^{(i)},g^{(i)},h^{(i)},e^{(i)},\alpha,\beta\}$, where $K(x,y)= \lambda \left({\begin{array}{cc}a(x,y) &d(x,y)\\b(x,y)&a(y,x)\end{array}} \right) (\lambda=\frac{1}{2}\, \mathrm{or}\, 1)$. Then for $k\in\mathbb{N}_+,$ $k\geqslant2,$
\begin{align*}
	&\frac{1}{2}  \int_{X ^k} \mathrm{Tr}\big(K (x_1,x_2)K (x_2,x_3)\cdots K (x_k,x_1)\big)  d\mu(x_1)\cdots d\mu(x_k)\\
	&\quad= \lambda\mathrm{Tr}\big(A^k\big) + \mathrm{Tr}\big((A-A^2)p_k(A)\big) + \sum\limits_{l\leqslant \frac{k}{2}} C_{Q_1,Q_2,\cdots,Q_l}\mathrm{Tr}(Q_1Q_2\cdots Q_l),
\end{align*}
where $p_k$ is a polynomial of degree $k-2$  which coefficients depending on  $k,\lambda,\alpha,\beta$, $C_{Q_1,Q_2,\cdots,Q_l}$ are constants related to $k,\lambda,\beta,N$ and each $Q_i$ is a rank-1 operator expressed of the following forms:
\begin{align*}%\label{form-of-Qi}
	D A^{\dag\,m} f^{(i)}\otimes A^{m^{\prime}}g^{(i)},\quad\mathrm{or}\quad
	h^{(i)}\otimes A^{m^{\prime}} e^{(i)}
\end{align*}
for some $m,m',i$ satisfying  $0\leqslant m,m^{\prime}\leqslant k$, $i =1,\cdots,N.$
	%$f^{(i)}^{(t)},g_j^{(t)},k_i^{(t)}, h_j^{(t)}$ are taken from Assumption~\ref{assump_of_th3.1_2}.
\end{lemma}
\begin{proof} By Lemma \ref{identity},  it follows that
\begin{eqnarray}\label{(12)}
    &  & \frac{1}{2} \int_{X ^k} \mathrm{Tr}\big(K (x_1,x_2)K (x_2,x_3)\cdots K (x_k,x_1)\big)  d\mu(x_1)\cdots d\mu(x_k)\nonumber \\
    &=&  \sum_{ 	i_2,\cdots,i_k\in\{1,2\}
		}
		\int_{X ^k} K_{  1,i_2}(x_1,x_2)K_{ i_2,i_3}(x_2,x_3)\cdots K_{ i_k,  1}(x_k,x_1)  d\mu(x_1)\cdots d\mu(x_k).\quad
\end{eqnarray}
	When the case that $(i_2,\cdots,i_k)= (1,\cdots,1)$,   we have that
%\begin{align*}
	%&
\begin{eqnarray}\label{(13)}\int_{X^k} K_{1,1}(x_1,x_2)K_{1,1}(x_2,x_3)\cdots K_{1,1}(x_k,x_1) d\mu(x_1)\cdots d\mu(x_k) =  \lambda^k\mathrm{Tr}(A^k).\end{eqnarray}
	%&\int_{X^k} K_{2,2}(x_1,x_2)K_{2,2}(x_2,x_3)\cdots K_{2,2}(x_k,x_1) d\mu(x_1)\cdots d\mu(x_k) = %\lambda^k\mathrm{Tr}(A^{\dag\,k})=\lambda^k\mathrm{Tr}(A^k).
%\end{align*}
We now consider the integral in the case that $(i_1,\cdots,i_k)\neq (1,\cdots,1).$  Noting that the sequence
  $\{(i_1,i_2), (i_2,i_3), \cdots, (i_k,i_1)\}$ doesn't contain consecutive occurrences of $``(1,2)"$ or $``(2,1)"$, we may rewrite
\begin{align}\label{terms-of-T_k}
	&\int_{X^k} K_{ 1,i_2}(x_1,x_2)K_{i_2,i_3}(x_2,x_3)\cdots K_{i_k, 1}(x_k,x_1) dx_1\cdots dx_k\nonumber\\
 =& \lambda^k\mathrm{Tr}\Big(A^{n_1}D A^{\dag\,n_2}BA^{n_3}\cdots A^{n_{2m-1}}D A^{\dag\,n_{2m}}B\Big),
\end{align}
with some nonnegative integer $n_1,\cdots,n_{2m}$  satisfying $n_1+\cdots+n_{2m-1}+n_{2m}+2m = k.$
By the assumption    $A^{\dag} B -B  A = \sum\limits_{i=1}^{N}  f^{(i)}\otimes g^{(i)}$ and $D B -(\alpha A^{2}+\beta A)  = \sum\limits_{i=1}^{N} h^{(i)}\otimes e^{(i)}$,   we have that for an nonnegative integer $n\leqslant k$,
\begin{align}\label{(9)}
	D A^{\dag\,n}B
	&= DBA^{n} + \sum\limits_{j=1}^n D A^{\dag\,n-j} (A^{\dag}B-BA)A^{j-1} \nonumber\\
	&=DBA^{n} + \sum\limits_{j=1}^n \sum\limits_{i=1}^{N}  D A^{\dag\,n-j} (f^{(i)}\otimes g^{(i)})A^{j-1}\nonumber\\
    &= (\alpha A^{2}+\beta A)A^n +\sum\limits_{i=1}^{N}( h^{(i)}\otimes e^{(i)}) A^{n}+ \sum\limits_{j=1}^n \sum\limits_{i=1}^{N}  D A^{\dag\,n-j} (f^{(i)}\otimes g^{(i)})A^{j-1}.
		%		&=A^{n+2} + (DB-A^2)A^{n}
		%		+\sum\limits_{j=1}^n D A^{\dag\,n-j} (AB-BA)A^{j-1},
\end{align}
Substituting  \eqref{(9)} into \eqref{terms-of-T_k}, we have that
\begin{align}\label{(11)}
&\int_{X^k} K_{ 1,i_2}(x_1,x_2)K_{i_2,i_3}(x_2,x_3)\cdots K_{i_k, 1}(x_k,x_1) dx_1\cdots dx_k\nonumber\\
 =& \lambda^k\mathrm{Tr}\Big(A^{k-2m}(\alpha A^{2}+\beta A)^m \Big)+ \sum\limits_{l\leqslant \frac{k}{2}} C_{Q_1,Q_2,\cdots,Q_l}\mathrm{Tr}(Q_1Q_2\cdots Q_l)
\end{align}
with some constant  $C_{Q_1,Q_2,\cdots,Q_l}.$ Combing it with \eqref{(12)}, \eqref{(13)}, we have that
\begin{align}\label{(14)}
& \frac{1}{2} \int_{X ^k} \mathrm{Tr}\big(K (x_1,x_2)K (x_2,x_3)\cdots K (x_k,x_1)\big)  d\mu(x_1)\cdots d\mu(x_k)\nonumber\\
 =& \lambda^k\sum\limits_{m=0}^{[\frac{k}{2}]} \tbinom{k}{2m} \mathrm{Tr}\big(A^{k-2m}(\alpha A^2+\beta A)^m \big)+ \sum\limits_{l\leqslant \frac{k}{2}} C_{Q_1,Q_2,\cdots,Q_l}\mathrm{Tr}(Q_1Q_2\cdots Q_l).
\end{align}
with some constant  $C_{Q_1,Q_2,\cdots,Q_l} $ depending on $\lambda,k,N,\alpha$.

Moreover,  a direct computation shows that
%\begin{align}\label{(15)}
%& \lambda^k\sum\limits_{m=0}^{[\frac{k}{2}]} \tbinom{k}{2m}  \big(A^{k-2m}(\alpha A^2+\beta A)^m \big)
%= &
%\end{align}
%
%
%\begin{align*}
	%DB=DB-(\alpha A^2+\beta A) + A^2 + \beta(A-A^2),
%\end{align*}
%and
\begin{align}\label{(15)}
	&\lambda^k\sum\limits_{m=0}^{[\frac{k}{2}]} \tbinom{k}{2m} \mathrm{Tr}\big(A^{k-2m}(\alpha A^2+\beta A)^m)\big)\nonumber\\
 =& \lambda^k\sum\limits_{m=0}^{[\frac{k}{2}]} \tbinom{k}{2m} \mathrm{Tr}\big(A^{k-2m}((\alpha+\beta )A^2+\beta (A-A^2))^m)\big)\nonumber\\
 =&\lambda^k\sum\limits_{m=0}^{[\frac{k}{2}]} \tbinom{k}{2m}  (\alpha+\beta )^m \mathrm{Tr}\big(A^{k } \big)+\mathrm{Tr}\big((A-A^2) p_k(A) \big)\nonumber\\
	=&\frac{(1+\sqrt{\alpha+\beta})^k+(1-\sqrt{\alpha+\beta})^k}{2}\lambda^k  \mathrm{Tr}\big(A^{k } \big)+\mathrm{Tr}\big((A-A^2) p_k(A) \big).
\end{align}
Here  $p_k$ is a polynomial of degree $k-2$  which coefficients depending on  $k,\lambda,\alpha,\beta$. When $\lambda=1 $ and $\alpha+\beta=0,$ or
 $\lambda=\frac{1}{2} $ and  $\alpha+\beta=1,$ it is clear that
 $$\frac{(1+\sqrt{\alpha+\beta})^k+(1-\sqrt{\alpha+\beta})^k}{2}\lambda^k=\lambda. $$
Substitution it and \eqref{(15)} into \eqref{(14)}, we derive Lemma~\ref{claim-of-tr-inconsistent}.
%\begin{align*}
	%DB= DB-(\alpha A^2+\beta A) +\beta(A-A^2),
%\end{align*}
%and
%\begin{align*}
%	&\lambda^k\sum\limits_{m=0}^k \tbinom{k}{2m} \mathrm{Tr}\big(A^{k-2m}(\alpha A^2+\beta A)^m)\big)\\
	%&\quad=\lambda^k\sum\limits_{m=0}^k \tbinom{k}{2m} \mathrm{Tr}\big(A^{k-2m}(\beta(A-A^2))^m)\big)\\
	%&\quad=\mathrm{Tr}(A^k) + \sum\limits_{m=1}^k \tbinom{k}{2m} \mathrm{Tr}\big(A^{k-2m}(\beta(A-A^2))^m)\big).
%\end{align*}
%Define a polynomial $p_k$ as follows:
%\begin{align}\label{eq-poly-pk}
	%p_k(x) =  \left\{
	%\begin{array}{cc}
		%\lambda^k\sum\limits_{m=1}^{[\frac{k}{2}]} \sum\limits_{j=1}^m\tbinom{k}{2m}\tbinom{m}{j}\beta^{j}
		%x^{k-2j}(x-x^2)^{j-1},  & \quad \mathrm{if}\,\lambda=\frac{1}{2}, \\
		%\lambda^k\sum\limits_{m=1}^{[\frac{k}{2}]} \tbinom{k}{2m}\beta^{m} x^{k-2m}(x-x^2)^{m-1}, & \quad %\mathrm{if}\,\lambda=1.
	%\end{array}
	%\right.
%\end{align}
%Since $A$ is self-adjoint, for $i=1,\cdots,N,$ any $\xi\in L^2(X,\mu),$
	%\begin{align*}
	%	(f^{(i)}\otimes g^{(i)})A \xi = \langle A\xi,g^{(i)} \rangle f^{(i)} = \langle \xi,Ag^{(i)} \rangle f^{(i)}=\big(f^{(i)}\otimes (A %g^{(i)})\big) \xi.
	%\end{align*}
	%Then we have
	%\begin{align*}
		%(f^{(i)}\otimes g^{(i)})A= f^{(i)} \otimes (A g^{(i)}),
	%\end{align*}
	%and similarly, $(h^{(i)}\otimes e^{(i)})A= h^{(i)}\otimes (A e^{(i)}).$
%By leveraging the self-adjoint property of $A_t,\widetilde{A_t}$,
%we derive Claim~\ref{claim-of-tr-inconsistent} from %Assumption~\ref{assump_of_th3.1_2}.
%\end{observation}
\end{proof}
\begin{remark}
The proof of the lemma shows that in the FRCR definition, the term $\alpha A^2+\beta A$ can be replaced by any polynomial $p(A)$ with uniformly bounded coefficients, where the sum of coefficients satisfies: $1$ when $\lambda=\frac{1}{2}$, and $0$ when
$\lambda=0$. Theorem \ref{Th2} remains valid under this modification.

This observation also indicates that the argument developed in this section does not apply to other values of $\lambda$.
\end{remark}
 We  remark the following lemma to estimate the trace of the term involving $Q_i$.
\begin{lemma}\label{claim-inner-prod-DAf-Ag}%修改一下
	Under the conditions of Theorem~\ref{Th2}, for a fixed integer $t $ greater than $2$, take $Q_{1,L},\cdots,Q_{t,L}$ from
	\begin{align*}
		D_L A_L^{\dag\,m} f^{(i)}_L\otimes A_L^{m^{\prime}}g^{(i)}_L,\quad\mathrm{or}\quad
		h^{(i)}_L\otimes A_L^{m^{\prime}} e^{(i)}_L,\quad i=1,\cdots,N, 1\leqslant m,m'\leqslant 2t+1.
	\end{align*}	
	As $L\rightarrow+\infty,$ we have that
	\begin{align*}
		\mathrm{Tr}\big(Q_{1,L}Q_{2,L}\cdots Q_{t,L}\big) = o\big(\mathrm{Var}_{\mathbb{P}_L}(\#_{X_L})\big)^{t}.
	\end{align*}
\end{lemma}
\begin{proof}Write $Q_{j,L}=p_{j,L}\otimes q_{j,L}~(j=1,\cdots,t)$, where $p_{j,L}, q_{j,L}$ are taken from $D_L A_L^{\dag\,m} f^{(i)}_L$,
$A_L^{m^{\prime}}g^{(i)}_L$, $h^{(i)}_L$, $A_L^{m^{\prime}} e^{(i)}_L$ for some $ i=1,\cdots,N, 1\leqslant m,m'\leqslant 2t+1.$ By the condition of Theorem~\ref{assump_of_th3.1_2}, one sees that for any $i,j$,
$$\langle p_{i,L},q_{j,L} \rangle=o\big(\mathrm{Var}_{\mathbb{P}_L}(\#_{X_L})\big). $$
It follows that
	%Noticing that for any rank-1 operators $f_1\otimes g_1,f_2\otimes g_2, \cdots, f_k\otimes g_k$ on $L^2(X_L)$, $f_i,g_i\in %L^2(X_L),$ $k\in\mathbb{N}_+,$ we have
	\begin{align*}
 \mathrm{Tr}\big(Q_{1,L}Q_{2,L}\cdots Q_{t,L}\big)&=
		\mathrm{Tr}\big((p_{1,L}\otimes q_{1,L})(p_{2,L}\otimes q_{2,L}) \cdots (p_{t,L}\otimes q_{t,L})\big) \\
 &= \langle p_{1,L},q_{t,L} \rangle\langle p_{2,L},q_{1,L}\rangle\cdots \langle p_{t,L},q_{t-1,L} \rangle=o\big(\mathrm{Var}_{\mathbb{P}_L}(\#_{X_L})\big)^{t},
	\end{align*} which completing the proof.
	%It suffices to show that
	%for $m,m^{\prime}\in \mathbb{Z}_{\geqslant 0},$ $i,j=1,\cdots,N,$
	%\begin{align*}%\label{estimate-of-assump4}
		%\langle D_L A_L^{{\dag}\, m} f^{(i)}_{L},  {A_L}^{ n} g^{(j)}_{L}\rangle,
		%\langle D_L A_L^{{\dag}\, m} f^{(i)}_{L},  {A_L}^{ n} e^{(j)}_{L}\rangle,
		%\langle h^{(i)}_{L}, {A_t}^{ n} g^{(j)}_{L}\rangle,
		%\langle h^{(i)}_{L},  {A_t}^{ n} e^{(j)}_{L}\rangle  = o\big(\mathrm{Var}_{\mathbb{P}_L}(\#_{X_L})\big),
	%\end{align*}
	%which is a condition of Theorem~\ref{assump_of_th3.1_2}.
\end{proof}

%\begin{lemma}\label{lem-norm-p_k(A)}%{\bf compute the coefficient involving  $a_L$ and $b_L$ in Theorem \ref{Th2}.}
%Suppose that $\sup\limits_L (\|A_L\|+\|A^{\dag}_L\|+|\alpha_L|+|\beta_L|)<M,$ where $M>1,$ then for $k\in\mathbb{N}_+,$ %$k\geqslant2,$
%\begin{align*}
%	\|p_{k,L}(A_L)\| = O(1) \quad\mathrm{as}\quad L\to +\infty.
%\end{align*}
%\end{lemma}
%\begin{proof}
%If $\lambda_L=\frac{1}{2},$
%\begin{align*}
%\|p_{k,L}(A_L)\|
	%&=\big\|\lambda_L^k\sum\limits_{m=1}^{[\frac{k}{2}]} \sum\limits_{j=1}^m\tbinom{k}{2m}\tbinom{m}{j}\beta_L^{j} %A_L^{k-2j}(A_L-A_L^2)^{j-1}\big\| \\
%	&\leqslant \lambda_L^k\sum\limits_{m=1}^{[\frac{k}{2}]} \sum\limits_{j=1}^m\tbinom{k}{2m}\tbinom{m}{j}M^{j} %M^{k-2j}(M^2+M)^{j-1}\\
	%&=\lambda_L^k\sum\limits_{m=1}^{[\frac{k}{2}]} \sum\limits_{j=1}^m\tbinom{k}{2m}\tbinom{m}{j} %M^{k-1}(M+1)^{j-1}\\
	%&\leqslant (\frac{1}{2})^{k}\frac{M^{k-1}}{M+1}
	%\big((M+2)^{\frac{1}{2}}+1\big)^k.
%\end{align*}
%If $\lambda_L=1,$
%\begin{align*}
%\|p_{k,L}(A_L)\|
%	&=\big\|\lambda_L^k\sum\limits_{m=1}^{[\frac{k}{2}]} \tbinom{k}{2m}\beta_L^{m} A_L^{k-2m}(A_L-A_L^2)^{m-1}\big\| \\
%	&\leqslant \lambda_L^k\sum\limits_{m=1}^{\frac{k}{2}} \tbinom{k}{2m}M^{m} M^{k-2m}(M+M^2)^{m-1}\\
%	&\leqslant \frac{M^{k-1}}{M+1}
	%\big((M+1)^{\frac{1}{2}}+1\big)^k.
%\end{align*}
%\end{proof}
%By the above theorem $\lambda=\frac{1}{2} or 1$ is necessary.

Now we turn to the proof of Theorem \ref{Th2}.

\noindent {\bf The proof of  Theorem \ref{Th2}.} We begin the proof with the argument for $V_{k}(L)$. By Lemmas~\ref{claim-of-tr-inconsistent} and \ref{claim-inner-prod-DAf-Ag}, we have that for a fixed positive integer $k$,
\begin{align*}%\label{eq-V_k(L)}
V_k(\#_{X_L }) &=\frac{1}{2}  \int_{X_L ^k} \mathrm{Tr}\big(K_L (x_1,x_2)K_L (x_2,x_3)\cdots K_L (x_k,x_1)\big) d\mu_L (x_1)\cdots d\mu_L (x_k) \nonumber\\
	&\quad= \lambda_L\mathrm{Tr}\big(A_L^k\big) + \mathrm{Tr}\big((A_L-A_L^2)p_{k,L}(A_L)\big) + o\big(\mathrm{Var}_{\mathbb{P}_L}(\#_{X_L})\big)^{\frac{k}{2}},
\end{align*}
where $p_{k,L}$ is the polynomial appeared in Lemmas~\ref{claim-of-tr-inconsistent}.  From the assumption  $$\|A_L-A_L^2\|_1
	 =o\big(\mathrm{Var}_{\mathbb{P}_L}(\#_{X_L})\big)^\delta, $$
one sees that
\begin{align*}%\label{eq-V_k(L)}
V_k(\#_{X_L })  =   \lambda_L\mathrm{Tr}\big(A_L) + o\big(\mathrm{Var}_{\mathbb{P}_L}(\#_{X_L})\big)^\delta + o\big(\mathrm{Var}_{\mathbb{P}_L}(\#_{X_L})\big)^{\frac{k}{2}}.
\end{align*}
Combing it with   \eqref{eq-cn-and-Vk},  we have that for %any integer $n$ greater than $2\delta+1$,
%\begin{align*}
	%c_n(\#_{X_L})&=\sum\limits_{k=2}^{n}(-1)^{k-1}(k-1)! v(n-1,k-1)(V_{k}(L)-V_{k-1}(L)) \\
  %&=o\big(\mathrm{Var}_{\mathbb{P}_L}(\#_{X_L})\big)^{\frac{n}{2}}.
%\end{align*}
%By the definition of cumulants, cumulants
%$$c_2(\#_{X_L}) = \mathrm{Var}_{\mathbb{P}_L}[\#_{X_L}] = V_1(L)-V_2(L).$$
%For $k\geqslant3,$ by \eqref{eq-V_k(L)},
%\begin{align*}%\label{eq2-of-lem-induction-Vk}
	%V_k(L)-V_{k-1}(L)&= \lambda_L\mathrm{Tr}\big(A_L^k-A_L^{k-1}\big) + %\mathrm{Tr}\big((A_L-A_L^2)p_{k,L}(A_L)\big)-
	%\mathrm{Tr}\big((A_L-A_L^2)p_{k-1,L}(A_L)\big)\\
	%&\quad+o\big(\mathrm{Var}_{\mathbb{P}_L}(\#_{X_L})\big)^{\frac{k}{2}}.
%\end{align*}
%Using the inequality $|\mathrm{Tr}(AB)| \leqslant \|A\|_1\|B\|$ for trace-class operator $A$ and bounded linear operator $B$, under %conditions of \eqref{assump_of_th3.1_2} and using Lemma~\ref{lem-norm-p_k(A)}, for $k\in\mathbb{N}_+,$ $k\geqslant2$,
%\begin{align*}
	%|\mathrm{Tr}\big((A_L-A_L^2)p_{k,L}(A_L)\big)|
	%&\leqslant \|A_L-A_L^2\|_1\|p_{k,L}(A_L)\|\\
	%&=o\big(\mathrm{Var}_{\mathbb{P}_L}(\#_{X_L})\big)^{\delta}.
%\end{align*}
%For $k\geqslant3,$
%\begin{align*}
	%|\mathrm{Tr}\big(A_L^k-A_L^{k-1}\big)|\leqslant \|A_L-A_L^2\|_1 %\|A_L\|^{k-2}o\big(\mathrm{Var}_{\mathbb{P}_L}(\#_{X_L})\big)^{\delta}.
%\end{align*}
%Hence for $k\geqslant3,$
%\begin{align*}
	%|V_k(L)-V_{k-1}(L)| = o\big(\mathrm{Var}_{\mathbb{P}_L}(\#_{X_L})\big)^{\delta}+ %o\big(\mathrm{Var}_{\mathbb{P}_L}(\#_{X_L})\big)^{\frac{k}{2}}.
%\end{align*}
%Then
  $ n>\max\{2\delta,3\},$
\begin{align*}
	c_n(\#_{X_L})&=\sum\limits_{k=2}^{n}(-1)^{k-1}(k-1)! v(n-1,k-1)(V_{k}(L)-V_{k-1}(L))\\
	&= v(n-1,1)\mathrm{Var}_{\mathbb{P}_L}(\#_{X_L})+ o\big(\mathrm{Var}_{\mathbb{P}_L}(\#_{X_L})\big)^{\delta}+ o\big(\mathrm{Var}_{\mathbb{P}_L}(\#_{X_L})\big)^{\frac{n}{2}}\\
 &= o\big(\mathrm{Var}_{\mathbb{P}_L}(\#_{X_L})\big)^{\frac{n}{2}}.
\end{align*}
%and for $$
%\begin{align*}
%c_n(\#_{X_L}),
%\end{align*}
The conclusion of Theorem~\ref{Th2} now follows from Lemma \ref{CLT-and-cn} and the above reasoning.\hfill $\Box$

\section{CLT for the Pfaffian $\mathrm{Sine}_4$  and $\mathrm{Sine}_1$ processes}\label{section-CLT-for-sine-beta}

In this section, we investigate the central limit theorems for the Pfaffian $\mathrm{Sine}_4$ and $\mathrm{Sine}_1$ processes. We begin by analyzing their properties, then proceed to verify their central limit theorems through characteristic function analysis  to explain the basic strategy.

\subsection{The Pfaffian $\mathrm{Sine}_4$-process}\label{CLT-for-sine4} Recall that
the Pfaffian $\mathrm{Sine}_4$-process $\mathbb{P}_{\mathrm{Sine}_4}$ is a Pfaffian point process on $\mathbb{R}$
with the matrix kernel $\mathbb{K}_{\mathrm{Sine}_4}$  with respect to the Lebesgue measure $dx$, where
\begin{align*}
	\mathbb{K}_{\mathrm{Sine}_4}(x,y)= ZK_{\mathrm{Sine}_4}(x,y),\quad
	K_{\mathrm{Sine}_4}(x,y)
	= \frac{1}{2}
	\left(\begin{array}{cc}
		S(x-y) & S^{\prime}(x-y)  \\
		IS(x-y) & S(x-y)
	\end{array}\right),
\end{align*}
while
\begin{align*}
	S(x) = \frac{\sin (\pi x) }{ \pi x}, \quad  IS(x) = \int_0^x S(t) dt.
\end{align*}
%It is
 Clearly  the functions $S$, $IS$ and $S^{\prime}$ are all bounded functions on $\mathbb{R}$.   Furthermore,
It is well known that
\begin{align*}
	S(\xi) = \frac{\sin (\pi \xi) }{ \pi \xi} = \big(\chi_{[-\frac{1}{2},
		\frac{1}{2}]}\big)^{\textasciicircum}(\xi) = \int_{\mathbb{R}} e^{-i 2\pi t \xi } \chi_{[-\frac{1}{2},\frac{1}{2}]}(t) dt.
\end{align*}
Therefore,  applying Plancherel theorem shows that $S\in L^2(\mathbb{R})$ and $\|S\|_{L^2(\mathbb{R})} = 1$.
%so that
%\begin{align*}
%     \|A\| = \frac{\lambda}{2}\big\|\widehat{S}\big\|_{L^{\infty}(\mathbb{R})} = \frac{\lambda}{2},
%     \quad
%     \|D\| = \frac{\lambda}{2} \big\|\widehat{S^{\prime}}\big\|_{L^{\infty}(\mathbb{R})} = \frac{\lambda\pi}{2}.
%\end{align*}
%We refer the reader to \cite{DysonFreemanJ.1970, Forrester2010} for more details and background about the %$\mathrm{Sine}_4$-process.
% bufetov2019number, Bufetov2021, Soshnikov2000fluctuations

 For $L>0$, let $I_L = (-L,L)$ and denote by $\#_L$  the number of the Pfaffian ${\mathrm{Sine}_4}$ process in $I_L.$  A straightforward calculation yields
%\begin{align*}
	%\mathbb{E}_{\mathrm{Sine}_4} [\#_L]= \frac{1}{2}\int_{I_L} S(0) dx = L,\quad
	%\mathrm{Var}_{\mathrm{Sine}_4} (\#_L) =\frac{1}{2\pi^2}\log L + O(1),
%\end{align*}
%and
%\begin{align*}
   % \frac{\#_t-\mathbb{E}_{\mathrm{Sine}_4} [\#_t] }{\sqrt{\mathrm{Var}_{\mathrm{Sine}_4}(\#_t)}} % \xrightarrow[t\rightarrow+\infty]{d} N(0,1).
%\end{align*}
%indeed, we have
\begin{align*}
    \mathbb{E}_{\mathrm{Sine}_4} [\#_L]= \frac{1}{2}\int_{I_L} S(0) dx = L.
\end{align*}
Following \cite{Costin1995}, we obtain the asymptotic variance as
\begin{align*}
\mathrm{Var}_{\mathrm{Sine}_4} (\#_L)
&= \frac{1}{2}\int_{I_L}S(0)dx-\int_{I_L^2}\mathrm{det}K_{{\mathrm{Sine}_4}}(x,y)dx dy\\
&= \frac{1}{2}\int_{I_L}S(0)dx-\frac{1}{4}\int_{I_L^2} [S^2(x-y)-S'(x-y) IS(x-y)] dx dy\\
&=\frac{1}{2}\int_{I_L}S(0)dx-\frac{1}{2}\int_{I_L^2}S^2(x-y)dx dy+\frac{1}{4}IS^2(2L)\\
&=\frac{1}{2\pi^2}\log L + O(1).
\end{align*}
The third equality follows from integration by parts applied to functions $S'(x-y)$ and $IS(x-y) $.

We next check the  FRCP for the Pfaffian ${\mathrm{Sine}_4}$ process.
\begin{proposition}\label{FRCPsin4}
For any $L>0,$   the Pfaffian ${\mathrm{Sine}_4}$ process on  $I_L$ has the finite rank commutator
property.
\end{proposition}
\begin{proof}
Let $A_L=A_L^{\dag}, B_L, D_L$ be the   integral operators on $L^2(I_L)$  with the integral kernel
\begin{align*}
    &A_L(x,y) = A_L^{\dag}(x,y) =\chi_{I_L}(x)S(x-y)\chi_{I_L}(y),\\
    &B_L(x,y) = \chi_{I_L}(x)IS(x-y)\chi_{I_L}(y),\\
    &D_L(x,y) = \chi_{I_L}(x)S^{\prime}(x-y)\chi_{I_L}(y),
\end{align*}
respectively.
It follows from the Fourier transform for $S(x),S'(x)$  that  $0\leqslant A_L\leqslant1,$ $\|D_L\|\leqslant \pi.$ However, $\{B_L\}_{L>0}$ are not uniformly bounded.

From the standard argument for the   integral operators,  the commutate  $A_LB_L-B_LA_L$ is also a integral operator with
the   integral kernel
\begin{align*}
    (A_LB_L-B_LA_L)(x,y)
    &= \chi_{I_L}(x) \,\chi_{I_L}(y)\,\int_{I_L} \big(S(x-z)IS(z-y)-IS(x-z)S(z-y)\big) dz\\
    &=\big(IS(L-x)IS(L-y)-IS(-L-x)IS(-L-y)\big)\chi_{I_L}(x)\chi_{I_L}(y).
    \end{align*}
  This implies that the commutate
   \begin{equation}\label{commutate_ALBL}
     A_LB_L-B_LA_L  =   IS(L-\cdot)\chi_{I_L}(\cdot) \otimes IS(L-\cdot)\chi_{I_L}(\cdot)- IS(L+\cdot)\chi_{I_L}(\cdot) \otimes IS(L+\cdot)\chi_{I_L}(\cdot) \\
   \end{equation}
    is a rank-two operator on $L^{2} (I_L)$.

Similarly, the integral operator $D_LB_L-A_L^2$ has the integral kernel
\begin{align*}
     (D_LB_L-A_L^2)(x,y)
    &=\chi_{I_L}(x)\chi_{I_L}(y) \int_{I_L} \big(S^{\prime}(x-z)IS(z-y)-S(x-z)S(z-y)\big) dz \\
    &=\big(-S(L-x)IS(L-y)+S(-L-x)IS(-L-y)\big)\chi_{I_L}(x)\chi_{I_L}(y).
\end{align*}
It follows that
   \begin{equation}\label{commutate_DLBL}
     D_LB_L- A_L^2  =   -S(L-\cdot)\chi_{I_L}(\cdot)\otimes IS(L-\cdot)\chi_{I_L}(\cdot)-  S(L+\cdot)\chi_{I_L}(\cdot)\otimes IS(L+\cdot)\chi_{I_L}(\cdot) \\
   \end{equation}
    is also a rank-two operator on $L^{2} (I_L)$, completing the proof of the proposition.
\end{proof}
To verify the remaining conditions of Theorem~\ref{Th2}, we remark the following lemma.
%By Assumption \ref{assump_of_th3.1_2} and it suffices to prove the following claim.
%\begin{claim}\label{inner-prod-sine4}
	%For any $m,m^{\prime}\in \mathbb{Z}_{\geqslant0},$ and for  %$f^{(t)},g^{(t)},k^{(t)}\in\big\{(IS(L\pm\cdot)-\frac{1}{2})\chi_{I_L},\chi_{I_L}\big\}$, $h^{(t)} = S(L\pm\cdot)\chi_{I_L}$,
	%when $t\rightarrow+\infty$,
	%\begin{align}\label{estimate-of-assump4}
		%&\langle D_L A_t^m f^{(t)}, A_t^{m^{\prime}} g^{(t)}\rangle_{L^2(\mathbb{R})},\quad
		%\langle D_L A_t^m f_i^{(t)}, A_t^{m^{\prime}} k^{(t)}\rangle_{L^2(\mathbb{R})}  = O(1),\\\nonumber
		%&\langle h^{(t)}, A_t^{m^{\prime}} g^{(t)}\rangle_{L^2(\mathbb{R})},\quad
		%\langle h^{(t)}, A_t^{m^{\prime}} k^{(t)}\rangle_{L^2(\mathbb{R})} = O(1).
	%\end{align}
%\end{claim}

\begin{lemma}\label{lemma-an-easy-condition-for-Th}
%Generally, \eqref{equ-in-assup4}(iv) is difficult to verify. In some cases, there are easier conditions. For example,
Fixed an integer $N$. Let $\{\mathbb{P}_L\}_{L\geqslant 0}$ be a family of Pfaffian point processes having FRCP with $\{N,f^{(i)}_{L},g^{(i)}_{L},h^{(i)}_{L},e^{(i)}_{L},\alpha_L,\beta_L\}$ with the matrix-valued kernel function $\mathbb{K}_L(x,y)=ZK_L(x,y),$
	$$K_L(x,y)= \lambda_L \left({\begin{array}{cc}a_L(x,y) &d_L(x,y)\\b_L(x,y)&a_L(y,x)\end{array}} \right) (\lambda_L=\frac{1}{2}\, or\, 1),$$
over the domain $X_L$ .  Suppose  that
 $\mathrm{Var}_{\mathbb{P}_L} (\#_{X_L}) \rightarrow +\infty$ as $L\rightarrow +\infty$, and
  $A_L,B_L,D_L,A^{\dag}_L$ are bounded operators on $L^2(X_L)$  and $M=\sup\limits_L (\|A_L\|+\|A^{\dag}_L\|+\|D_L\|)<+\infty.$

If there are the  decompositions
\begin{align*}
f_L^{(i)} = f_{L,1}^{(i)}+ f_{L,2}^{(i)},\quad
g_L^{(i)} = g_{L,1}^{(i)}+ g_{L,2}^{(i)},\quad
e_L^{(i)} = e_{L,1}^{(i)}+ e_{L,2}^{(i)}, \,i=1,\cdots, N
\end{align*}
in  $L^2(X_L)$  such that
\begin{align*}
\{f_{L,2}^{(i)},~g_{L,2}^{(i)},~e_{L,2}^{(i)},~h_{L}^{(i)},~
A_L^{\dag} f_{L,1}^{(i)}-f_{L,1}^{(i)},~
A_L g_{L,1}^{(i)}-g_{L,1}^{(i)},~
A_L e_{L,1}^{(i)}-e_{L,1}^{(i)},~
D_Lf_{L,1}^{(i)}: i=1,\cdots, N,L\geqslant 0\}
\end{align*}
are uniformly  bounded by $M^{\prime}$ with respect to $L^2$-norm,    and
\begin{align*}
\langle D_L f_{L,1}^{(i)},g_{L,1}^{(j)}\rangle,~
\langle D_L f_{L,1}^{(i)},e_{L,1}^{(j)}\rangle,~
\langle h_L^{(i)}, g_{L,1}^{(j)}\rangle,~
\langle h_L^{(i)}, e_{L,1}^{(j)}\rangle
= o\big(\mathrm{Var}_{\mathbb{P}_L}(\#_{X_L})\big),
\end{align*}then the statement  of  Theorem \ref{Th2} $\mathrm{(iv)}$ holds.
\end{lemma}
\begin{proof}
By the assumption of the Lemma, we have that
  for nonnegative integer   $m,n$
\begin{align*}
	D_L A_L^{{\dag}\, m} f^{(i)}_{L}
	&= D_L A_L^{{\dag}\, m} (f^{(i)}_{L,1} + f^{(i)}_{L,2})\\
	%&= D_L A_L^{{\dag}\, m-1} (A_Lf^{(i)}_{L,1} - f^{(i)}_{L,1}+ f^{(i)}_{L,1})+D_L A_L^{{\dag}\, m}f^{(i)}_{L,2} \\
	&=D_L f_{L,1}^{(i)}+D_L A_L^{{\dag}\, m}f^{(i)}_{L,2}+ \sum\limits_{s=0}^{m-1} D_L A_L^{{\dag}\,s} (A_Lf^{(i)}_{L,1} - f^{(i)}_{L,1}),
\end{align*}
and
$$A_L^{n} g^{(j)}_{L} = g_{L,1}^{(j)}+ A_L^{n}g^{(j)}_{L,2}+ \sum\limits_{s=0}^{n-1} A_L^{s} (A_Lg^{(j)}_{L,1} - g^{(j)}_{L,1}). $$
It is easy to see that
\begin{align*}
	\|D_L f_{L,1}^{(i)}\|\leqslant M^{\prime},\quad
	\big\| D_L A_L^{{\dag}\, m}f^{(i)}_{L,2}+ \sum\limits_{s=0}^{m-1} D_L A_L^{{\dag}\,s} (A_Lf^{(i)}_{L,1} - f^{(i)}_{L,1})\big\|\leqslant M^{m+1}M^{\prime}+m M^{s+1}M^{\prime},
\end{align*}
and
$$\big\|A_L^{n}g^{(j)}_{L,2}+ \sum\limits_{s=0}^{n-1} A_L^{s} (A_Lg^{(j)}_{L,1} - g^{(j)}_{L,1})\big\|\leqslant M^nM'+nM^sM'.$$
This implies that
\begin{align*}
&  \langle D_L A_L^{{\dag}\, m} f^{(i)}_{L},A_L^{n} g^{(j)}_{L}\rangle\\=&\langle D_L f_{L,1}^{(i)}+[D_L A_L^{{\dag}\, m}f^{(i)}_{L,2}+ \sum\limits_{s=0}^{m-1} D_L A_L^{{\dag}\,s} (A_Lf^{(i)}_{L,1} - f^{(i)}_{L,1})], g_{L,1}^{(j)}+ [ A_L^{n}g^{(j)}_{L,2}+ \sum\limits_{s=0}^{n-1} A_L^{s} (A_Lg^{(j)}_{L,1} - g^{(j)}_{L,1})]\rangle\\
=&\langle D_L f_{L,1}^{(i)},g_{L,1}^{(j)}\rangle+O(1)\\
=& o\big(\mathrm{Var}_{\mathbb{P}_L}(\#_{X_L})\big).
\end{align*}
The same argument yields analogous asymptotics for the other terms in Theorem \ref{Th2} (iv), completing the proof.
%Calculated that for $n\in\mathbb{Z}_{\geqslant0},$
%\begin{align*}
	%&A_L^{n} g^{(j)}_{L} = g_{L,1}^{(j)}+ A_L^{n}g^{(j)}_{L,2}+ \sum\limits_{s=0}^{n-1} A_L^{s} (A_Lg^{(j)}_{L,1} - %g^{(j)}_{L,1}),\\
	%&A_L^{n} e^{(j)}_{L} = e_{L,1}^{(j)}+ A_L^{n}e^{(j)}_{L,2}+ \sum\limits_{s=0}^{n-1} A_L^{s} (A_Le^{(j)}_{L,1} - %e^{(j)}_{L,1}),
%\end{align*}
%and
%\begin{align*}
	%\|A_L^{n}g^{(j)}_{L,2}\|,\|A_L^{n}e^{(j)}_{L,2}\|\leqslant M^nM^{\prime},\quad \|A_L^{s} (A_Lg^{(j)}_{L,1} - %g^{(j)}_{L,1})\|,\|A_L^{s} (A_Le^{(j)}_{L,1} - e^{(j)}_{L,1})\|\leqslant M^{s+1}M^{\prime}.
%\end{align*}
%Then we obtain the lemma.
\end{proof}
To verify the decomposition required by Lemma \ref{lemma-an-easy-condition-for-Th} for the Pfaffian ${\mathrm{Sine}_4}$ process, the following crucial  observation is essential.
\begin{lemma}\label{lemma-2-norm-of-IS-1/2} We have that
%For any $t>0,$
%\begin{align*}
    $\big\|(IS-\frac{1}{2})\chi_{(0,+\infty)}\big\|_{L^2(\mathbb{R})}\leqslant 1.$
%\end{align*}
\end{lemma}
\begin{proof} When $0< x\leqslant 1$, it is clear that
  $IS(x) = \int_0^x S(u) du\leqslant 1.$
When  $x>1,$  by the fact that
\begin{align*}
     \int_0^{+\infty} S(x)dx =   \int_0^{+\infty} \frac{\sin(\pi x)}{\pi x}dx = \frac{1}{2},
\end{align*}
we have that
\begin{align*}
    \frac{1}{2} - IS(x)
    &=  \int_x^{+\infty} \frac{sin(\pi u)}{\pi u} du\\
    &= \int_x^{\lfloor x\rfloor+1} \frac{sin(\pi u)}{\pi u} du + \sum\limits_{k=1}^{+\infty} \int_{\lfloor x\rfloor+k}^{\lfloor x\rfloor+k+1} \frac{sin(\pi u)}{\pi u} du.\\
    &\leqslant \int_x^{\lfloor x\rfloor+1} \big|\frac{sin(\pi u)}{\pi u}\big| du + \int_{\lfloor x\rfloor+1}^{\lfloor x\rfloor+2}  \big|\frac{sin(\pi u)}{\pi u}\big| du\\
    &\leqslant \frac{2}{\pi x},
\end{align*}
 where the inequality in the third line originates from the treatment of alternating series. This implies that
\begin{align*}
    \big\|(IS-\frac{1}{2})\chi_{(0,+\infty)}\big\|_{L^2(\mathbb{R})}^2
    &= \Big(\int_0^1 + \int_1^{+\infty}\Big) \big|IS(x)-\frac{1}{2})\big|^2 dx\\
    &\leqslant \frac{1}{4} + \int_1^{+\infty}\big(\frac{2}{\pi x}\big)^2dx
     \leqslant 1,
\end{align*}
which attains Lemma~\ref{lemma-2-norm-of-IS-1/2}.
\end{proof}

% begin{example}\label{example-of-th2}

We now proceed to establish the central limit theorem for the Pfaffian  $\mathrm{Sine4}$ process  (Theorem \ref{Th2}). Conditions (i) and (ii) of Theorem \ref{Th2} are clearly satisfied.
From \cite{Costin1995,Soshnikov2000fluctuations}, we have the asymptotic estimate as $L\to+\infty,$
$$\|A_L-A_L^2\|_1=O(\log L).$$
To complete the proof, it remains to verify condition (iv) of Theorem \ref{Th2}. This follows directly from Lemma \ref{lemma-an-easy-condition-for-Th}.

Indeed, by considering the Fourier transform of
$S$, we observe that
$$\big\|S(L\pm\cdot) \big\|_{L^2(\mathbb{R})} =\bigg[\int_{\mathbb R}\big|\frac{sin \pi x}{\pi x}\big|^2 dx\bigg]^{\frac{1}{2}}=1,$$
Furthermore, Lemma~\ref{lemma-2-norm-of-IS-1/2} yields $$ \Big\|\big(IS(L\pm\cdot)-\frac{1}{2}\big)\chi_{I_L}\Big\|_{L^2(\mathbb{R})} \leqslant 1.$$
It follows that
\begin{align*}
    \|D_L(\chi_{I_L})\|_{L^2(\mathbb{R})}
    &=\|\chi_{I_L}(\cdot)\int_{I_L}S^{\prime}(\cdot-y)dy\,\|_{L^2(\mathbb{R})}\\
    &=\|\chi_{I_L}(\cdot) \big(-S(L-\cdot)+S(L+\cdot)\big)\,\|_{L^2(\mathbb{R})}\leqslant 2, \end{align*}
    and
 \begin{align*}   \|A_L(\chi_{I_L}) -\chi_{I_L}\| _{L^2(\mathbb{R})}
    &=\big\|(\int_{I_L}S(\cdot-y)dy-1)\,\chi_{I_L}(\cdot)\big\|_{L^2(\mathbb{R})}\\
    &=\|\big(IS(L-\cdot)+IS(L+\cdot)-1\big)\chi_{I_L}(\cdot)\|_{L^2(\mathbb{R})}\leqslant 2.%,\\
   % &=\Big(\big(IS(L+x)-\frac{1}{2}\big)\chi_{I_L}(x) + \big(IS(L-x)-\frac{1}{2}\big)\chi_{I_L}(x) \Big) +\chi_{I_L}(x),
\end{align*}
Furthermore,  for any $  L\geqslant 0$
\begin{align*}\big\langle D_L\chi_{I_L},\chi_{I_L}\big\rangle_{L^2(\mathbb{R})}&=
\big\langle \chi_{I_L}(\cdot) \big(-S(L-\cdot)+S(L+\cdot)\big),\chi_{I_L}\big\rangle_{L^2(\mathbb{R})}\\
& %\big\langle S(L\pm\cdot)\chi_{I_L},\chi_{I_L}\big\rangle_{L^2(\mathbb{R})}
= \int_{-
L}^L [S(L+ x)-S(L-x)] dx \\
&= 0.
\end{align*}
%and furthermore, for $i\in\mathbb{N}_+,$
%\begin{align*}
   % \|A_L^i(\chi_{I_L})-\chi_{I_L}\| = \sum\limits_{j=0}^{i-1}A_L^j \Big[\big(IS(L+\cdot)-\frac{1}{2}\big)\chi_{I_L} + % \big(IS(L-\cdot)-\frac{1}{2}\big)\chi_{I_L} \Big](x) + \chi_{I_L}(x),
%\end{align*}
%and
%\begin{align}\label{D-tA-t-chi-I}
    % D_LA_L^i(\chi_{I_L})(x) &= \sum\limits_{j=0}^{i-1} D_LA_L^j \Big[\big(IS(L+\cdot)-\frac{1}{2}\big)\chi_{I_L} + % % \big(IS(L-\cdot)-\frac{1}{2}\big)\chi_{I_L} \Big](x) \\
     %&\quad + \big(S(L+x)-S(L-x)\big)\chi_{I_L}(x).\nonumber
%\end{align}
%An estimate exists in the following claim; its proof will be given later.
%
%Here we have estimates $\big\|S(L\pm\cdot)\chi_{I_L}\big\|_{L^2(\mathbb{R})}\leqslant 1$
%and $\Big\|\big(IS(L\pm\cdot)-\frac{1}{2}\big)\chi_{I_L}\Big\|_{L^2(\mathbb{R})} \leqslant 1$
%by Lemma~\ref{lemma-2-norm-of-IS-1/2}.
%By \eqref{D-tA-t-chi-I}, for $i\in \mathbb{Z}_{\geqslant0},$
%\begin{align*}
  %  \|D_LA_L^i(\chi_{I_L})\|_{L^2(\mathbb{R})}\leqslant 2\pi i + 2.
%\end{align*}
By \eqref{commutate_ALBL} and \eqref{commutate_DLBL},  the Pfaffian $\mathrm{Sine}_4$ process over $I_L$ has the FRCR with $(2, f_L^{(i)}, g_L^{(i)}, h_L^{(i)},e_L^{(i)}, 1,0)$, where
$$f_L^{(1)}=g_L^{(1)}=e_L^{(1)}=\chi_{I_L}(\cdot)IS(L-\cdot), \,\, -f_L^{(2)}=g_L^{(2)}=e_L^{(2)}=\chi_{I_L}(\cdot) IS(L+\cdot)$$
and
$$h_L^{(1)}=-\chi_{I_L}(\cdot)S(L-\cdot), \,\, h_L^{(2)}=- \chi_{I_L}(\cdot)S(L+\cdot).$$

Therefore, by the decomposition
$$ IS(L\pm\cdot)\chi_{I_L}(\cdot)=(IS(L\pm\cdot)-\frac{1}{2})\chi_{I_L}(\cdot)+\frac{1}{2}\chi_{I_L}(\cdot)$$
and the argument in Lemma \ref{lemma-an-easy-condition-for-Th}, we have that
for any $i,j = 1,2 ,$ and $m,n\in \mathbb{Z}_{\geqslant0},$
$$
	\langle D_L A_L^{{\dag}\, m} f^{(i)}_{L},  {A_L}^{ n} g^{(j)}_{L}\rangle,
	\langle D_L A_L^{{\dag}\, m} f^{(i)}_{L},  {A_L}^{ n} e^{(j)}_{L}\rangle,
	\langle h^{(i)}_{L}, {A_t}^{ n} g^{(j)}_{L}\rangle,
	\langle h^{(i)}_{L},  {A_t}^{ n} e^{(j)}_{L}\rangle  = O(1)
$$
as $L\to+\infty$, which verify condition (iv) of Theorem \ref{Th2}.
%
%Then for $f_L^{(i)}=(IS(L\pm\cdot)-\frac{1}{2})\chi_{I_L}\,\mathrm{or}\,\chi_{I_L},$ $g_L^{(j)},e_L^{(j)} = %(IS(L\pm\cdot)-\frac{1}{2})\chi_{I_L}$ and $h_L^{(i)} =S(L\pm\cdot)\chi_{I_L}$, every term of \eqref{equ-in-assup4} (iv) tends %to $O(1)$
%According to the above calculations and Lemma~\ref{lemma-an-easy-condition-for-Th}, we only need the following estimates:
%\begin{align*}
%\big\langle S(L\pm\cdot)\chi_{I_L},\chi_{I_L}\big\rangle_{L^2(\mathbb{R})}
%= \int_{-
%L}^L S(L\pm x) dx = IS(2L) = O(1)\quad \mathrm{as}\quad L\to+\infty,
%\end{align*}
%for $i\in\mathbb{Z}_{\geqslant0},$
%\begin{align*}
%&\Big\langle D_LA_L^i\Big[\big( IS(L\pm\cdot)-\frac{1}{2}\big)\chi_{I_L}\Big],\chi_{I_L}\Big\rangle_{L^2(\mathbb{R})}\\
%&\quad=\Big\langle A_L^i\Big[\big( %IS(L\pm\cdot)-\frac{1}{2}\big)\chi_{I_L}\Big],-D_L(\chi_{I_L})\Big\rangle_{L^2(\mathbb{R})}\\
%&\quad= \Big\langle A_L^i\Big[\big( IS(L\pm\cdot)-\frac{1}{2}\big)\chi_{I_L}\Big], %\frac{1}{2}\big(S(L-\cdot)-S(L+\cdot)\big)\chi_{I_L}\Big\rangle_{L^2(\mathbb{R})}\\
%&\quad=O(1) \quad \mathrm{as}\quad L\rightarrow+\infty,
%\end{align*}
%and for $i\in \mathbb{Z}_{\geqslant0},$
%\begin{align*}
%\langle D_LA_L^i(\chi_{I_L}),\chi_{I_L}\rangle_{L^2(\mathbb{R})}
%&=\langle A_L^i(\chi_{I_L}),-D_L(\chi_{I_L})\rangle_{L^2(\mathbb{R})}\\
%&= \Big\langle\big(\frac{1}{2}\big)^i\chi_{I_L}, %\frac{1}{2}\big(S(L-\cdot)-S(L+\cdot)\big)\chi_{I_L}\Big\rangle_{L^2(\mathbb{R})} + O(1)\\
%&=O(1) \quad \mathrm{as}\quad t\rightarrow+\infty.
%\end{align*}

Applying Theorem~\ref{Th2}, we obtain the CLT for the Pfaffian $\mathrm{Sine}_4$-process over $I_L$. That is, as $L\to+\infty$,
\begin{align*}
	\frac{\#_{I_L}-\mathbb{E}_{\mathrm{Sine}_4} [\#_{I_L}] }{\sqrt{\mathrm{Var}_{\mathrm{Sine}_4}(\#_{I_L})}} \xrightarrow{d} N(0,1).
\end{align*}

%\end{proof}

\subsection{The Pfaffian $\mathrm{Sine}_1$ process}
We now briefly discuss the central limit theorem for the Pfaffian $\mathrm{Sine}_1$ process, as its proof closely follows that of the Pfaffian $\mathrm{Sine}_4$-process.

Recall that the Pfaffian  $\mathrm{Sine}_1$ process, denoted $\mathbb{P}_{\mathrm{Sine}_1}$,  is a Pfaffian point process on $\mathbb{R}$
with the matrix kernel $\mathbb{K}_{\mathrm{Sine}_1}$ (with respect to the Lebesgue measure
$dx$), where
\begin{align*}
	\mathbb{K}_{\mathrm{Sine}_1}(x,y)= ZK_{\mathrm{Sine}_1}(x,y),\quad
	K_{\mathrm{Sine}_1}(x,y)
	=
	\left(\begin{array}{cc}
		S(x-y) & S^{\prime}(x-y)  \\
		IS(x-y)-\frac{1}{2}\sgn(x-y) & S(x-y)
	\end{array}\right).
\end{align*}

For $L>0$, let $I_L = (-L,L)$ and $\#_L = \#_{\mathrm{Sine}_1}I_L$ . A straightforward computation yields
$$\mathbb{E}_{\mathrm{Sine}_1} [\#_L]=\int_{I_L} S(0) dx = 2L.$$
Moreover, by \cite{Costin1995}, as $L\to+\infty$,  the variance satisfies
\begin{align*}
	\mathrm{Var}_{\mathrm{Sine}_1} (\#_L)
	&= \int_{I_L}S(0)dx-\int_{I_L^2}\mathrm{det}K_{{\mathrm{Sine}_1}}(x,y)dx dy\\
	&=2L-2\int_{I_L^2}S(x-y)^2dx dy+IS(2L)^2-IS(2L)\\
	&=\frac{2}{\pi^2}\log L + O(1).
\end{align*}
%\begin{align*}
%	\frac{\#_t-\mathbb{E}_{\mathrm{Sine}_1} [\#_t] }{\sqrt{\mathrm{Var}_{\mathrm{Sine}_1}(\#_t)}} \xrightarrow[t\rightarrow+\infty]{d} N(0,1).
%\end{align*}

For $L\geqslant 0,$ denote $A_L=A_L^{\dag}, B_L,D_L$ be the integral operators with the integral kernels
\begin{align*}
	&A_L(x,y) = A_L^{\dag}(x,y) = \chi_{I_L}(x)S(x-y)\chi_{I_L}(y),\\
	&B_L(x,y) = \chi_{I_L}(x)\big(IS(x-y)-\frac{1}{2}\sgn(x-y)\big)\chi_{I_L}(y),\\
	&D_L(x,y) = \chi_{I_L}(x)S^{\prime}(x-y)\chi_{I_L}(y),
\end{align*}
respectively.
By the argument in Subsection 3.1, we have that $\|A_L\|\leqslant1,$ $\|D_L\| \leqslant \pi$ and
	$$\|A_L-A_L^2\|_1=O(\log L) $$
as $L\to+\infty.$

A careful computation shows that $A_LB_L-B_LA_L$, $D_LB_L+A_L-A_L^2$ are the integral operators with the integral kernels
\begin{align*}
	(A_LB_L-B_LA_L)(x,y)
	 =\chi_{I_L}(x)\chi_{I_L}(y)\big(IS(L-x)IS(L-y)-IS(L+x)IS(L+y) \\
	  +\frac{1}{2} ( IS(L+x) -IS(L-x)+ IS(L+y) - IS(L-y) \big),
\end{align*}
and
\begin{align*}
	(D_LB_L+A_L-A_L^2)(x,y)
	 =-\chi_{I_L}(x)\chi_{I_L}(y)\big(S(L-x)IS(L-y)+S(L+x)IS(L+y) \\
	 +\frac{1}{2} (S(L-x)+S(L+x)\big),
\end{align*}
respectively.
It follows that the Pfaffian  $\mathrm{Sine}_1$-process over $I_L$ have FRCP with   $(4,  f_L^{(i)},  g_L^{(i)}, h_L^{(i)}, e_L^{(i)},  1, 1)$, where
$$f_L^{(1)}=g_L^{(1)}= \chi_{I_L}(\cdot)IS(L-\cdot), \,\, -f_L^{(2)}=g_L^{(2)}= \chi_{I_L}(\cdot)IS(L+\cdot),$$
$$f_L^{(3)}=\frac{1}{2} \chi_{I_L}(\cdot)( IS(L+\cdot) -IS(L-\cdot)),  g_L^{(3)}=\chi_{I_L}(\cdot);  f_L^{(4)}=\chi_{I_L}(\cdot), g_L^{(4)}=\frac{1}{2}\chi_{I_L}(\cdot) ( IS(L+\cdot) -IS(L-\cdot));$$
and
$$h_L^{(1)}=- \chi_{I_L}(\cdot) S(L-\cdot), \,\, h_L^{(2)}=- \chi_{I_L}(\cdot)  S(L+\cdot),\,\, h_L^{(3)}=-\frac{1}{2} \chi_{I_L}(\cdot)  ( S(L-\cdot)+S(L+\cdot)), \,\,  h_L^{(4)} =0,$$
$$e_L^{(1)}=\chi_{I_L}(\cdot)IS(L-\cdot), \,\,   e_L^{(2)}=\chi_{I_L}(\cdot)IS(L+\cdot), \,\, e_L^{(3)}=\chi_{I_L}(\cdot),  \,\, e_L^{(4)}=0.$$
Using the decomposition
$$IS(L\pm\cdot)=\chi_{I_L}(\cdot)(IS(L\pm\cdot)-\frac{1}{2})+ \frac{1}{2} \chi_{I_L}(\cdot)$$
and the fact proved in the above subsection  that
$$\|D_L(\chi_{I_L})\|_{L^2(\mathbb{R})}, \|A_L(\chi_{I_L}) -\chi_{I_L}\| _{L^2(\mathbb{R})}\leqslant 2,$$
and
$$ \big\langle D_L\chi_{I_L},\chi_{I_L}\big\rangle_{L^2(\mathbb{R})}=0.$$
it follows from the argument of Lemma \ref{lemma-an-easy-condition-for-Th} that
%Moreover, we have
%\begin{align*}
	%D_L(\chi_{I_L})(x)
	%&=\int_{I_L}S^{\prime}(x-y)dy\,\chi_{I_L}(x)\\
	%&=\big(-S(L-x)+S(L+x)\big)\chi_{I_L}(x), \\
	%A_L(\chi_{I_L})(x)
	%&=\int_{I_L}S(x-y)dy\,\chi_{I_L}(x)\\
	%&=\big(IS(L-x)+IS(L+x)\big)\chi_{I_L}(x),\\
	%&=\Big(\big(IS(L+x)-\frac{1}{2}\big)\chi_{I_L}(x) + \big(IS(L-x)-\frac{1}{2}\big)\chi_{I_L}(x) \Big) +\chi_{I_L}(x).
%\end{align*}
%Using the calculations in Section~\ref{CLT-for-sine4},
for $   i,j=1,\cdots,4$ and nonnegative integers $m,m'$,
\begin{align*}
	&\langle D_L A_L^m f_L^{(i)}, A_L^{m^{\prime}} g^{(j)}_L\rangle_{L^2(\mathbb{R})},\quad
		\langle D_L A_L^m f_L^{(i)}, A_L^{m^{\prime}} e^{(j)}_L\rangle_{L^2(\mathbb{R})}  = O(1),\\
	&\langle h_L^{(i)}, A_L^{m^{\prime}} g^{(j)}_L\rangle_{L^2(\mathbb{R})},\quad
		\langle h_L^{(i)}, A_L^{m^{\prime}} e^{(j)}_L\rangle_{L^2(\mathbb{R})} =O(1).
\end{align*}as $L\to+\infty,$  which verify condition (iv) of Theorem \ref{Th2}.

Then the conditions of \eqref{equ-in-assup4} are verified.
Applying Theorem~\ref{Th2}, we obtain the central limit theorem for the Pfaffian $\mathrm{Sine}_1$-process over $I_L$. That is, as $L\rightarrow+\infty$,
\begin{align*}
	\frac{\#_{I_L}-\mathbb{E}_{\mathrm{Sine}_1} [\#_{I_L}] }{\sqrt{\mathrm{Var}_{\mathrm{Sine}_1}(\#_{I_L})}} \xrightarrow{d} N(0,1).
\end{align*}

\subsection{CLT for step functions}
The remainder of this section is devoted to proving the central limit theorem for step functions (i.e., finite linear combinations of indicator functions)  in the context of the   Pfaffian ${\mathrm{Sine}_4}$ and ${\mathrm{Sine}_1}$ processes.

We begin with a general preparatory result.

For $L\geqslant  0$, let $\{\mathbb{P}_L\}_{L\geqslant 0}$ be a family of Pfaffian point processes with the matrix-valued kernel $\mathbb{K}_L(x,y)=ZK_L(x,y),$ where
$$K_L(x,y)=\lambda_L\left({\begin{array}{cc}a_L(x,y) &d_L(x,y)\\b_L(x,y)&a_L(y,x)\end{array}} \right)( \lambda_L=\frac{1}{2}~\text{or}~1) ,$$
%(\lambda_L=\frac{1}{2}\, or\, 1)
defined over the domain $X_L$ with respect to the measure $\mu_L$.

Fix an positive integer $N$.  Let $\varphi_L = \sum\limits_{i = 1}^N\lambda_i\chi_{I_{L}^{(i)}}$ be a step function on $X_L$, where $I_{L}^{(1)},\cdots,I_{L}^{(N)}$ are disjoint bounded Borel subsets of $X_L.$ Denote $$I_L=\mathrm{supp} \varphi_L, $$ and define the following integral operators with their respective kernels:
$$A_L:=\text{operator with kernel } \chi_{I_L}(x)a_L(x,y)\chi_{I_L}(y),$$
$$B_L:=\text{operator with kernel } \chi_{I_L}(x)b_L(x,y)\chi_{I_L}(y),$$
$$D_L:=\text{operator with kernel } \chi_{I_L}(x)d_L(x,y)\chi_{I_L}(y),$$
$$A^{\dag}_L:=\text{operator with kernel } \chi_{I_L}(x)a_L(y,x)\chi_{I_L}(y).$$
We require the following assumption:
\begin{assumption}\label{assump-CLT-stepfunc}
$\,$
\begin{itemize}
\item[(i)] $\mathrm{Var}_{\mathbb{P}_L} (S_{\varphi_L}) \rightarrow +\infty$ as $L\rightarrow +\infty;$
\item[(ii)] $A_L$ belong to trace class. Furthermore,  there exists $\delta>0$ such that for $i=1,\cdots,N,$
	\begin{align*}
		\big\|\chi_{I_{L}^{(i)}}A_L\chi_{I_{L}^{(i)}}-(\chi_{I_{L}^{(i)}}A_L\chi_{I_{L}^{(i)}})^2\big\|_1= O\big(\mathrm{Var}_{\mathbb{P}_L}(S_{\varphi_L})\big)^{\delta}\quad\mathrm{as}\quad L\rightarrow+\infty,
	\end{align*}
	and for any pair  $i,j\in\{1,\cdots,N\} $ with $i\neq j,$
	\begin{align*}
		 \big\|\chi_{I_{L}^{(i)}}A_L\chi_{I_{L}^{(j)}}A_L^{*}\chi_{I_{L}^{(i)}}\big\|_1 = O\big(\mathrm{Var}_{\mathbb{P}_L}(S_{\varphi_L})\big)^{\delta}\quad\mathrm{as}\quad L\rightarrow+\infty;
	\end{align*}
\item[(iii)] $A_L,B_L,D_L,A^{\dag}_L$ are bounded operators and $$\sup\limits_L (\|A_L\|+\|A^{\dag}_L\|)<+\infty;$$
%\quad\sup\limits_{i,L}(|\alpha_L^{(i,i^{\prime})}|+|\beta_L^{(i,i^{\prime})}|)<+\infty;
\item[(iv)] There exists  an integer $M$   such that for any $i=1,\cdots,N$, $A_L^{\dag}\chi_{I_{L}^{(i)}} B_L-B_L \chi_{I_{L}^{(i)}}A_L$ and $D_L\chi_{I_{L}^{(i)}}B_t-A_L\chi_{I_{L}^{(i)}}A_L$ are finite rank operators of rank at most $M$; that is,
\begin{align*}
	&A_L^{\dag}\chi_{I_{L}^{(i)}} B_L-B_L \chi_{I_{L}^{(i)}}A_L = \sum\limits_{j=1}^{M} f_{L}^{(i,j)}\otimes g_L^{(i,j)},\\
	&D_L\chi_{I_{L}^{(i)}}B_t-(\alpha_L A L\chi_{I_{L}^{(i)}}+\beta_L A_L\chi_{I_{L}^{(i)}}A_L) = \sum\limits_{j=1}^{M} h_L^{(i,j)}\otimes e_L^{(i,j)},
\end{align*}
where %$\alpha_L^{(i,i^{\prime})},\beta_L^{(i,i^{\prime})}\in\mathbb{C},\,\alpha_L^{(i,i^{\prime})} +\beta_L^{(i,i^{\prime})} =\begin{cases}
%		1, &  \mathrm{if}\, \lambda_L=\frac{1}{2},\\
%		0, &  \mathrm{if}\,\lambda_L=1,
%	\end{cases}$
 $f^{(i,j)}_{L},g_L^{(i,j)},h_L^{(i,j)},e_L^{(i,j)}\in L^2(X_L)$ for $i=1,\cdots,N,$ $j=1,\cdots,M$, and $\alpha,\beta\in\mathbb{C}$ with $\alpha_L +\beta_L =\begin{cases}
1, &  \mathrm{if}\, \lambda=\frac{1}{2};\\
0, &  \mathrm{if}\,\lambda=1.
\end{cases} $
and $\sup\limits_L |\alpha_L|+|\beta_L|<\infty;$
\item[(v)] For any $m,m^{\prime}\in \mathbb{N}_{+},$ $m\geqslant2,$  $i,i^{\prime},i_1,\cdots,i_m,j_1,\cdots,j_{m^{\prime}}\in\{1,\cdots,N\}\,$ and $j,j^{\prime}\in\{1,\cdots,M\},$ the following asymptotics hold:
\begin{align*}
	&\langle \chi_{I_L^{(i_1)}} D_L\chi_{I_L^{(i_2)}}A_L^{\dag}\chi_{I_L^{(i_3)}}\cdots A_L^{\dag}\chi_{I_L^{(i_m)}} f_L^{(i,j)},\quad \chi_{I_L^{(j_1)}}A_L\chi_{I_L^{(j_2)}}\cdots A_L\chi_{I_L^{( j_{m^{\prime}} )}} g_L^{(i^{\prime},j^{\prime})}\rangle,\\
   &\langle \chi_{I_L^{(i_1)}} D_L\chi_{I_L^{(i_2)}}A_L^{\dag}\chi_{I_L^{(i_3)}}\cdots A_L^{\dag}\chi_{I_L^{(i_m)}} f_L^{(i,j)},\quad \chi_{I_L^{(j_1)}}A_L\chi_{I_L^{(j_2)}}\cdots A_L\chi_{I_L^{( j_{m^{\prime}} )}} e_L^{(i^{\prime},j^{\prime})}\rangle,\\
   &\langle \chi_{I_L^{(i_1)}} h_L^{(i,j)},\quad \chi_{I_L^{(j_1)}}A_L\chi_{I_L^{(j_2)}}\cdots A_L\chi_{I_L^{( j_{m^{\prime}} )}} g_L^{(i^{\prime},j^{\prime})}\rangle,\\
   &\langle \chi_{I_L^{(i_1)}} h_L^{(i,j)},\quad \chi_{I_L^{(j_1)}}A_L\chi_{I_L^{(j_2)}}\cdots A_L\chi_{I_L^{( j_{m^{\prime}} )}} e_L^{(i^{\prime},j^{\prime})}\rangle\\
   &\quad =o\big(\mathrm{Var}_{\mathbb{P}_L}(S_{\varphi_L})\big), \quad\mathrm{as}\quad L\rightarrow+\infty.
\end{align*}	
\end{itemize}

\end{assumption}

\begin{theorem}\label{Th-step-func} For $L\geqslant  0$, let $\{\mathbb{P}_L\}_{L\geqslant 0}$ be a family of Pfaffian point processes with the matrix-valued kernel function $\mathbb{K}_L(x,y)=ZK_L(x,y) $
%$$K_L(x,y)= \frac{1}{2}\left({\begin{array}{cc}a_L(x,y) &d_L(x,y)\\b_L(x,y)&a_L(y,x)\end{array}} \right) ,$$
%(\lambda_L=\frac{1}{2}\, or\, 1)
over the domain $X_L$ with respect to the measure $\mu_L$.
Suppose that $\{\varphi_L = \sum\limits_{i = 1}^N\lambda_i\chi_{I_{L}^{(i)}}\}_{L\geqslant0}$ satisfy Assumption~\ref{assump-CLT-stepfunc}, then
\begin{align*}
	\frac{S_{\varphi_L} - \mathbb{E}_{\mathbb{P}_L} [S_{\varphi_L}] }{\sqrt{\mathrm{Var}_{\mathbb{P}_L}  (S_{\varphi_L})}}
\end{align*}				
converges in distribution to the standard normal law $N(0,1)$ as $L\rightarrow +\infty$.
\end{theorem}
We begin the proof  by the following lemma.
\begin{lemma}\label{claim-step-func-Tri_1A-i_kAi_1}
Under the assumptions of Theorem \ref{Th-step-func}, for $k\geqslant2$ and indices $i_1,\cdots,i_k\in\{1,\cdots,N\}$ where at least one pair   $(i_s, i_{s^{\prime}})$  satisfies  $i_s\neq  i_{s^{\prime}}$,
 the following asymptotic bound holds:
\begin{align*}
	\mathrm{Tr}(\chi_{I_L^{(i_1)}}A_L\chi_{I_L^{(i_2)}}A_L\cdots \chi_{I_L^{(i_k)}}A_L\chi_{I_L^{(i_1)}}) = O\big(\mathrm{Var}_{\mathbb{P}_L}(S_{\varphi_L})\big)^{\frac{\delta}{2}}\quad\mathrm{as}\quad L\rightarrow+\infty.
\end{align*}
\end{lemma}
\begin{proof} From Assumption  \ref{assump-CLT-stepfunc}(ii),
for any pair  $i,j\in\{1,\cdots,N\} $ with $i\neq j,$ we have that as $L\rightarrow+\infty$,
	\begin{align*} \|\chi_{I_{L}^{(i)}}A_L\chi_{I_{L}^{(j)}}\|_2\leqslant
		 \big\|\chi_{I_{L}^{(i)}}A_L\chi_{I_{L}^{(j)}}A_L^{*}\chi_{I_{L}^{(i)}}\big\|^{\frac{1}{2}}_1 = O\big(\mathrm{Var}_{\mathbb{P}_L}(S_{\varphi_L})\big)^{\frac{\delta}{2}}.
	\end{align*}
By the assumptions of the lemma, there exist two distinct indices $j,j'$ such that
$$i_{j}\neq i_{j+1} \,\,\text {and }\,\,i_{j'}\neq i_{j'+1},$$
 where we adopt the convention that $i_{k+1}=i_1$. Therefore,%
%
%Without loss of generalitzion,
%suppose that $i_1\neq i_2.$  A direct argument for operators of trace class shows that
%Using the Cauchy-Schwartz inequality
%$$\big|\mathrm{Tr}(AB)\big|^2\leqslant \mathrm{Tr}(AA^*)\mathrm{Tr}(BB^*), \forall \text{ Hilbert-Schmidt operators }A, B,$$   %we have that
\begin{align}\label{eq-TrI_1AAAAI_kA}
	&\quad ~\big|\mathrm{Tr}(\chi_{I_L^{(i_1)}}A_L\chi_{I_L^{(i_2)}}A_L\cdots \chi_{I_L^{(i_k)}}A_L\chi_{I_L^{(i_1)}})\big|\nonumber\\
	 & \leqslant \|\chi_{I_L^{(i_j)}}A_L\chi_{I_L^{(i_{j+1})}}\|_2 \|\chi_{I_L^{(i_{j'})}}A_L\chi_{I_L^{(i_{j'+1})}}\|_2
	  \|A_L\|^{k-1} \nonumber\\
&  =O\big(\mathrm{Var}_{\mathbb{P}_L}(S_{\varphi_L})\big)^{ {\delta} }
\end{align} as  $L\rightarrow+\infty$, completing the proof.
\end{proof}

We now turn to the proof of Theorem \ref{Th-step-func}.

\noindent{\bf{The proof of Theorem \ref{Th-step-func}. }}We employ the method of moments to establish the result. The argument follows closely the proof strategy of Theorem~\ref{Th2}, and therefore we provide only a concise outline of the main arguments.

By Lemma \ref{cumulants},  it  suffices to prove the existence of   a positive integer $n_0 $ such that for all integer $ n\geqslant n_0$,
 the following asymptotic estimate holds:
\begin{align*}
	c_n(S_{\varphi_L})= o\big( \mathrm{Var}_{\mathbb{P}_{L}} (S_{\varphi_L})\big)^{\frac{n}{2}} \quad\mathrm{as}\quad L\rightarrow+\infty.
\end{align*}
Applying Proposition~\ref{lemma-cumulants-Pfaffian} to the step function$\eta=\varphi_L = \sum\limits_{i = 1}^N\lambda_i\chi_{I_{L}^{(i)}}$, we obtain the following cumulant estimate.
%\begin{claim}\label{claim-cn-stepfunc}
For any positive integer $n,$ we have that
\begin{align}\label{cn}
	c_n(S_{\varphi_L})
	 = \sum\limits_{i=1}^N\lambda_i^n c_n(\#_{\mathbb{P}_L}I_{L}^{(i)})
	 + \sum_{\substack{2\leqslant k\leqslant n,\\
			i_1,\cdots,i_k\in \{1,\cdots,N\},\\
			\exists i_s,i_{s^\prime},~i_s\neq i_{s^\prime} }}
	C_{i_1,\cdots,i_k}
	\int_{I_L^{(i_1)}\times I_L^{(i_2)}\cdots\times I_L^{(i_k)}}\nonumber	\\
	 \mathrm{Tr}\big(K_L(x_1,x_2) \cdots K_L(x_k,x_1)\big)  d\mu_L(x_1)d\mu_L(x_2)\cdots d\mu_L(x_k),
\end{align}	
where $C_{i_1,\cdots,i_k}$ are constants related to $i_1,\cdots,i_k.$
%\end{claim}
By the argument of  Theorem~\ref{Th2}, for any $i=1,\cdots,N $  and integer $n >\max\{2\delta,3\},$
\begin{align}\label{estimate-c_n-Ii}
	c_n(\#_{\mathbb{P}_L}I_{L}^{(i)})= o\big( \mathrm{Var}_{\mathbb{P}_{L}} (S_{\varphi_L})\big)^{\frac{n}{2}} \quad\mathrm{as}\quad L\rightarrow+\infty.
\end{align}
%Similar to Claims \ref{claim-of-tr-inconsistent} and \ref{claim-inner-prod-DAf-Ag}, we obtain the following two claims.
%\begin{claim}\label{claim-step-func-Tr}
Using  the similar  argument in Lemma \ref{claim-of-tr-inconsistent}, we have that for $n \geqslant2,$
%For $t>0,$ let $A_t$ be the integral operator induced by $\chi_{\mathrm{supp}f_t}(x)S(x-y)\chi_{\mathrm{supp}f_t}(y),$ then when $t\rightarrow+\infty,$
	\begin{align*}
		&\quad\frac{1}{2}\int_{I_L^{(i_1)}\times I_L^{(i_2)}\cdots\times I_L^{(i_k)}} \mathrm{Tr}\big(K_L(x_1,x_2)K_L(x_2,x_3)\cdots K_L(x_k,x_1)\big) d\mu_L(x_1)d\mu_L(x_2)\cdots d\mu_L(x_k)\\
		&=\sum\limits_{(j_1,\cdots,j_m)\in J(i_1,\cdots,i_k)} C^L_{j_1,\cdots,j_m} \mathrm{Tr}(\chi_{I_L^{(j_1)}}A_L\cdots \chi_{I_L^{(j_m)}}A_L\chi_{I_L^{(j_1)}})+ \sum\limits_{l\leqslant \frac{k}{2}} C_{Q_{1,L},\cdots, Q_{l,L}}
		\mathrm{Tr}\big(Q_{1,L}\cdots Q_{l,L}\big),
	\end{align*}
where $J$ consists of all sequences obtained by optionally merging consecutive identical terms in the original sequence $(i_1,\cdots,i_k)$,  $C^L_{j_1,\cdots,j_m} $ are constants related to $\alpha_L$ and $\beta_L$,  $C_{Q_{1,L},Q_{2,L},\cdots, Q_{l,L}}$ are constants related to $i_1,\cdots,i_k,$ and each $Q_{i,L}$ is a rank-1 operator expressed of the following forms:
\begin{align*}%\label{form-of-Qi-stepfunc}
	&\big(\chi_{I_L^{(i_1)}} D_L\chi_{I_L^{(i_2)}}A_L^{\dag}\chi_{I_L^{(i_3)}}\cdots A_L^{\dag}\chi_{I_L^{(i_m)}} f_L^{(i,j)}\big)
	\otimes\big(\chi_{I_L^{(j_1)}}A_L\chi_{I_L^{(j_2)}}\cdots A_L\chi_{I_L^{( j_{m^{\prime}} )}} g_L^{(i^{\prime},j^{\prime})}\big),\nonumber\\
	&\mathrm{or}\quad\big(\chi_{I_L^{(i_1)}} h_L^{(i,j)}\big)
	\otimes \big(\chi_{I_L^{(j_1)}}A_L\chi_{I_L^{(j_2)}}\cdots A_L\chi_{I_L^{( j_{m^{\prime}} )}} e_L^{(i^{\prime},j^{\prime})}\big) \nonumber
\end{align*}
for some integers
   $m,m^{\prime} $ with $m\geqslant2,$  $i,i^{\prime},i_1,\cdots,i_m,j_1,\cdots,j_{m^{\prime}}\in \{1,\cdots,N\},$ and $j,j^{\prime}\in\{1,\cdots,M\}$.
%\end{claim}%$\begin{claim}\label{claim-step-func-TrQ1QQQ_l}
	%$For $l,k\in\mathbb{N}_+,$ $l\leqslant \frac{k}{2},$  % from \eqref{form-of-Qi-stepfunc},%\end{claim}

Following the argument in Lemma~\ref{claim-inner-prod-DAf-Ag}, each product term
$Q_{1,L} Q_{2,L} \cdots  Q_{l,L}$ satisfies the asymptotic estimate:
	\begin{align*}
		\mathrm{Tr}\big(Q_{1,L}Q_{2,L}\cdots Q_{k,L}\big) = o\big(\mathrm{Var}_{\mathbb{P}_L}(S_{\varphi_L})\big)^{\frac{k}{2}}\quad\mathrm{as}\quad L\rightarrow+\infty.
	\end{align*}
Consequently, we deduce the following:
\begin{align*}
	&\quad\frac{1}{2}\int_{I_L^{(i_1)}\times I_L^{(i_2)}\cdots\times I_L^{(i_k)}} \mathrm{Tr}\big(K_L(x_1,x_2)K_L(x_2,x_3)\cdots K_L(x_k,x_1)\big) d\mu_L(x_1)d\mu_L(x_2)\cdots d\mu_L(x_k)\\
	&=\sum\limits_{(j_1,\cdots,j_m)\in J(i_1,\cdots,i_k)} C^L_{j_1,\cdots,j_m} \mathrm{Tr}(\chi_{I_L^{(j_1)}}A_L\cdots \chi_{I_L^{(j_m)}}A_L\chi_{I_L^{(j_1)}})
	+ o\big( \mathrm{Var}_{\mathbb{P}_{t}} (S_{\varphi_L})\big)^{\frac{k}{2}}\\
&=o\big( \mathrm{Var}_{\mathbb{P}_{t}} (S_{\varphi_L})\big)^{\frac{k}{2}},
\end{align*}
where the final equality follows from Lemma \ref{claim-step-func-Tri_1A-i_kAi_1}.
Combining this result with equations     \eqref{cn} and \eqref{estimate-c_n-Ii},  , we obtain that for all integers $n> \max\{2\delta,3\},$ the cumulants satisfy the asymptotic estimate
\begin{align*}
	c_n(S_{\varphi_L}) = o\big( \mathrm{Var}_{\mathbb{P}_{L}} (S_{\varphi_L})\big)^{\frac{n}{2}}, \text{ as}~ L\rightarrow+\infty.
\end{align*}
 This completes the proof of the theorem.\hfill $\Box$

%\paragraph*{\bf CLT for the step function over the Pa $\mathrm{sine}_4$-process}
%\begin{example}\label{example-sine4-step-func}
%Let $\varphi= \sum\limits_{i=1}^N \lambda_i \chi_{(a_i,b_i)},$ where $\lambda_i\in\mathbb{R}\setminus \{0\}$, $(a_i,b_i)$ are %mutually disjoint intervals. For $L>0,$ let $\varphi_L(x) = \varphi(\frac{x}{L})$,
%then the centered normalization of $S_{\varphi_L}$ converges in distribution to $N(0,1),$ that is,
%\begin{align*}
%	\frac{\sum\limits_{i=1}^N\lambda_i\#_{\mathrm{Sine}_4}(a_iL,b_iL) - \frac{1}{2}\sum\limits_{i=1}^N\lambda_i(b_i-a_i)L %}{\sqrt{\mathrm{Var}_{\mathrm{Sine}_4}\Big(\sum\limits_{i=1}^N\lambda_i\#_{\mathrm{Sine}_4}(a_iL,b_iL)\Big)}} %\xrightarrow[L\rightarrow+\infty]{d} N(0,1).
%\end{align*}
%	
%\end{example}%We will use Theorem~\ref{Th-step-func} to prove it.
%In this section, we

We now proceed to establish the central limit theorem for the Pfaffian  $\mathrm{sine}_4$ and  $\mathrm{sine}_1$ processes. Owing to the near-identical structure of the proofs for both cases, we restrict our exposition to the Pfaffian $\mathrm{sine}_4$ process, with the adaptation to the $\mathrm{sine}_1$ case being straightforward.

Let $\varphi:\mathbb R\to\mathbb R$ be a compactly supported test function of the form
$$\varphi= \sum\limits_{i=1}^N \lambda_i \chi_{(a_i,b_i)},$$
where $\lambda_i\in\mathbb{R}\setminus \{0\}$, $(a_i,b_i)$
  is a collection of pairwise disjoint intervals. For any scaling parameter  $L>0,$ define the rescaled function $\varphi_L(x) = \varphi(\frac{x}{L})$.

We adopt the notation $I^{(i)}=(a_i,b_i)$ and $I_L^{(i)}=(a_iL,b_iL)$ for $i=1,\cdots N,$ $L>0.$ Without loss of generality, we assume the intervals are ordered such that $$a_1<b_1\leqslant a_2<b_2\leqslant\cdots\leqslant a_N<b_N.$$

%\begin{align*}
	%\varphi= \sum\limits_{i=1}^N \lambda_i \chi_{I^{(i)}},\quad \varphi_L=\sum\limits_{i=1}^N \lambda_i \chi_{I_L^{(i)}}.
%\end{align*}
For each scaling parameter $L>0$ , we consider the integral operators $A_L=A_L^{\dag}, B_L, D_L$ acting on $L^2(I_L)$, where
$I_L=(a_1L,b_NL).$  These operators are defined through their respective integral kernels:
\begin{align*}
    &A_L(x,y) = A_L^{\dag}(x,y) =\chi_{I_L}(x)S(x-y)\chi_{I_L}(y),\\
    &B_L(x,y) = \chi_{I_L}(x)IS(x-y)\chi_{I_L}(y),\\
    &D_L(x,y) = \chi_{I_L}(x)S^{\prime}(x-y)\chi_{I_L}(y).
\end{align*}
As established in previous analysis, these operators satisfy the following bounds:
 $$0\leqslant A_L\leqslant1,\,\,\|D_L\|\leqslant \pi.$$

The following lemma  verifies Assumption~\ref{assump-CLT-stepfunc}(i) for the Pfaffian  ${\mathrm{Sine}_4}$ process.
\begin{lemma}\label{lemma-sine4-Exp-var-stepfunc}
For the Pfaffian ${\mathrm{Sine}_4}$-process, the expectation and variance of the linear statistic $S_{\varphi_L}$ satisfy the following asymptotic behavior
	\begin{align*}
		&\mathbb{E}_{\mathrm{Sine}_4}[S_{\varphi_L}] = \frac{1}{2}\sum\limits_{i=1}^N\lambda_i(b_i-a_i)L,\\
		&\mathrm{Var}_{\mathrm{Sine}_4} (S_{\varphi_L})=
		O(\log L), 	\text{ as }  L\rightarrow+\infty,
	\end{align*}
where $\varphi_L(x) = \varphi(\frac{x}{L})$ is the rescaled test function of $\varphi= \sum\limits_{i=1}^N \lambda_i \chi_{(a_i,b_i)}.$

\end{lemma}
\begin{proof}%[Proof of Lemma~\ref{lemma-sine4-Exp-var-stepfunc}]
By equation \eqref{Math-Exp-For}, , we obtain the expectation:
\begin{align*}
	\mathbb{E}_{\mathrm{Sine}_4}[S_{\varphi_L}] = \frac{1}{2}\int_{\mathbb{R}} \varphi(x)S(0) dx =  \frac{1}{2}\sum\limits_{i=1}^N\lambda_i(b_i-a_i)L.
\end{align*}
For the variance calculation, applying  \eqref{Var-For} yields
\begin{align*}
	&\quad\mathrm{Var}_{\mathrm{Sine}_4} (S_{\varphi_L}) \\
	&=\frac{1}{2}\int_{\mathbb{R}}|\varphi_L(x)|^2S(0)dx-\int_{\mathbb{R}^2}\varphi_L(x)\overline{\varphi_L(y)}\mathrm{det}K_{{\mathrm{Sine}_4}}(x,y)dx dy\\
	& =\frac{1}{2}\int_{\mathbb{R}}|\varphi_L(x)|^2dx - \frac{1}{4}\int_{\mathbb{R}^2}\varphi_L(x)\overline{\varphi_L(y)}\big(S(x-y)^2-IS(x-y)S^{\prime}(x-y)\big)dxdy.
	%\frac{1}{2}\sum\limits_{i=1}^N\lambda_i^2(b_i-a_i)t
\end{align*}
Applying integration by parts to the cross-term in the variance expression yields
\begin{align*}
 &\quad\quad\int_{\mathbb{R}^2} \varphi_L(x)\overline{\varphi_L(y)} IS(x-y)S'(x-y) dxdy \\
 &= \sum_{i,j=1}^N \lambda_i\lambda_j \int_{I_L^{(i)} \times I_L^{(j)}} IS(x-y)S'(x-y) dxdy \\
& = -\int_{\mathbb{R}^2} \varphi_L(x)\overline{\varphi_L(y)} S(x-y)^2 dxdy
  + \frac{1}{2}\sum_{i,j=1}^N \lambda_i\lambda_j \Big[ IS((b_j-a_i)L)^2 + IS((a_j-b_i)L)^2
  \\&\quad\quad\quad\quad\quad\quad\quad\quad\quad\quad\quad\quad\quad\quad\quad\quad\quad\quad\quad   - IS((a_j-a_i)L)^2 - IS((b_j-b_i)L)^2 \Big].
\end{align*}

Given that  $IS\in L^{\infty}(\mathbb{R}),$    as $L $,  approaches infinity, the variance expression simplifies to
\begin{align}\label{substitute-Var-stepfunc}
	\mathrm{Var}_{\mathrm{Sine}_4} (S_{\varphi_L})
	=\frac{1}{2}\int_{\mathbb{R}}|\varphi_L(x)|^2dx - \frac{1}{2}\int_{\mathbb{R}^2}\varphi_L(x)\overline{\varphi_L(y)}S(x-y)^2dxdy+O(1).
\end{align}
The integrals evaluate as follows:
\begin{align}\label{v1}
	\int_{\mathbb{R}}|\varphi_L(x)|^2dx
	= \sum\limits_{i = 1}^N\lambda_i^2\int_{I_L^{(i)}} dx  =\sum\limits_{i=1}^N\lambda_i^2 (b_i-a_i)L,
\end{align}
and
\begin{align}\label{v2}
	&\quad\int_{\mathbb{R}^2}\varphi_L(x)\overline{\varphi_L(y)} S(x-y)^2 dxdy\nonumber\\
	&=\sum\limits_{i=1}^N\lambda_i^2\int_{I_L^{(i)\,2}} S(x-y)^2dxdy+2\sum\limits_{1\leqslant i<j\leqslant N}\lambda_i\lambda_j\int_{I_L^{(i)}\times I_L^{(j)}} S(x-y)^2dxdy.
\end{align}
Following \cite{Costin1995}, for any real numbers $a<b,$ as $L\rightarrow+\infty,$
\begin{align}\label{v3}
	\int_{(aL,bL)^2}S(x-y)^2dxdy = (b-a)L - \frac{1}{\pi^2}\log L +O(1).
\end{align}
It follows from \eqref{v3} that  in the case where $i<j$, $a_i<b_i=a_j<b_j,$
\begin{align}\label{v4}
	&\quad\int_{(a_iL,b_iL)\times (a_jL,b_jL)}S(x-y)^2dxdy \nonumber\\
&=\frac{1}{2}\int_{(a_iL,b_jL)^2}S(x-y)^2dxdy-\frac{1}{2}\int_{(a_iL,b_iL)^2}
S(x-y)^2dxdy-\frac{1}{2}\int_{(b_iL,b_jL)^2}S(x-y)^2dxdy \nonumber\\
	& =\frac{1}{2\pi^2}\log L +O(1)\quad\mathrm{as}\quad L\rightarrow +\infty.
\end{align}
For the case where $i<j$, $a_i<b_i<a_j<b_j,$
\begin{align}\label{v5}
	\int_{(a_iL,b_iL)\times (a_jL,b_jL)}S(x-y)^2dxdy
	&\leqslant\int_{(a_iL,b_iL)}\int_{(a_jL,b_jL)} \frac{1}{\big( \pi(a_j-b_i)L\big)^2} dydx \nonumber\\
	&=\frac{(b_i-a_i)(b_i-a_j)}{\pi^2(a_j-b_i)^2}=O(1).
\end{align}
Substituting these values \eqref{v1},\eqref{v2},\eqref{v3},\eqref{v4} and \eqref{v5}   into equation \eqref{substitute-Var-stepfunc}, we obtain the following result
%Substituting these into \eqref{substitute-Var-stepfunc}, it is obtained that when $t\rightarrow+\infty,$
\begin{align*}
	\mathrm{Var}_{\mathrm{Sine}_4} (S_{\varphi_L})
	&= \sum\limits_{i=1}^N \frac{\lambda_i^2}{2\pi^2}\log L -
	\sum_{\substack{1\leqslant i<j\leqslant N,\\
			b_i=a_j } }
	\frac{\lambda_i\lambda_j}{2\pi^2}\log L+O(1)\\
	&=O(Log L),
\end{align*}
completing the proof.
\end{proof}
We conclude by proving Theorem  \ref{Th-sine-beta-step-func}.

\noindent{\bf{The proof of Theorem \ref{Th-sine-beta-step-func}.}}
From \cite{Costin1995, Soshnikov2000fluctuations}, we obtain the asymptotic estimate as  $L\to+\infty,$
\begin{align*}
	\big\|\chi_{I_{L}^{(i)}}A_L\chi_{I_{L}^{(i)}}-(\chi_{I_{L}^{(i)}}A_L\chi_{I_{L}^{(i)}})^2\big\|_1= O(\log L)\quad\mathrm{as}\quad L\rightarrow+\infty,
\end{align*}
Furthermore, for any distinct $i,j\in\{1,\cdots,N\},$   by \eqref{v4} and \eqref{v5}
\begin{align*}
	\big\|\chi_{I_{L}^{(i)}}A_L\chi_{I_{L}^{(j)}}A_L^{*}\chi_{I_{L}^{(i)}}\big\|_1
	&=\int_{(a_iL,b_iL)\times (a_jL,b_jL)}S(x-y)^2dxdy\\
	&= O(\log L)\quad\mathrm{as}\quad L\rightarrow+\infty.
\end{align*}
Thus, Assumption~\ref{assump-CLT-stepfunc} (ii) is satisfied.

Analogous to Proposition \ref{FRCPsin4}, a direct computation reveals that $A_L \chi_{I_L^{(i)}}B_L -B_L \chi_{I_L^{(i)}}A_L$, $D_L\chi_{I_L^{(i)}}B_L -A_L \chi_{I_L^{(i)}} A_L$
 are integral operators with the following kernels:
\begin{align*}
	& \quad (A_L\chi_{I_L^{(i)}}B_L-B_L\chi_{I_L^{(i)}}A_L)(x,y)\\
	& =\chi_{(a_1L,b_NL) }(x) \chi_{(a_1L,b_NL) }(y)\int_{I_L^{(i)}} \big(S(x-z)IS(z-y)-IS(x-z)S(z-y)\big) dz \\
	  &= \big(IS(x-b_iL)IS(y-b_iL)-IS(x-a_iL)IS(y-a_iL)\big)\chi_{(a_1L,b_NL)}(x)\chi_{(a_1L,b_NL)}(y) ,
\end{align*}
and
\begin{align*}
	&\quad (D_L\chi_{I_L^{(i)}}B_L-A_L\chi_{I_L^{(i)}}A_L)(x,y)\\
	& =\chi_{(a_1L,b_NL) }(x) \chi_{(a_1L,b_NL) }(y) \int_{I_L} \big(S^{\prime}(x-z)IS(z-y)-S(x-z)S(z-y)\big) dz  \\
	& =\big(S(x-b_iL)IS(y-b_iL)-S(x-a_iL)IS(y-a_iL)\big)\chi_{(a_1L,b_NL)}(x)\chi_{(a_1L,b_NL)}(y).
\end{align*}
These identities confirm Assumption~\ref{assump-CLT-stepfunc} (iv).

By Lemma   \ref{lemma-2-norm-of-IS-1/2} and using the following decomposition
\begin{align*}
	&\quad IS(x-a_iL)\chi_{(a_1L,b_NL)}(x) \\
	&= (IS(x-a_iL)-\frac{1}{2})\chi_{(a_iL,b_NL)} -  (IS(a_iL-x)-\frac{1}{2})\chi_{(a_1L,a_iL)}
	+\frac{1}{2}\chi_{(a_iL,b_NL)}-\frac{1}{2}\chi_{(a_1L,a_iL)},
\end{align*}
an argument similar to that in Lemmas \ref{lemma-an-easy-condition-for-Th} verifies Assumption \ref{assump-CLT-stepfunc} (v).

By applying Theorem~\ref{Th-step-func}, we establish a central limit theorem for step functions in the Pfafian $\mathrm{Sine}_4$-process.
The corresponding result for the Pfaffian $\mathrm{Sine}_1$ process follows analogously, thereby completing the proof.


\begin{thebibliography}{99}
%已用
\bibitem{Aslaksen2001}
H. Aslaksen.
Quaternionic determinants.
Math. Intelligencer 18 (1996), no. 3, 57-65.

%Leble引用的
\bibitem{BLS2018}
F. Bekerman, T. Leblé, S. Serfaty.
CLT for fluctuations of $\beta$-ensembles with general potential.
Electron. J. Probab. 23 (2018), Paper no. 115, 31 pp.

%已用
\bibitem{Borodin_2005}
A. Borodin, E. M. Rains.
Eynard-Mehta theorem, Schur process, and their Pfaffian analogs.
J. Stat. Phys. 121 (2005), no. 3-4, 291-317.

%双正交,用到它的方法
\bibitem{Breuer&Duits2017}
J. Breuer, M. Duits.
Central limit theorems for biorthogonal ensembles and asymptotics of recurrence coefficients.
J. Amer. Math. Soc. 30 (2017), no. 1, 27-66.



%没用到
%\bibitem{Bufetov2018}
%A. I. Bufetov,
%Quasi-symmetries of determinantal point processes.
%Ann. Probab. 46 (2018), no. 2, 956-1003.

%已用
\bibitem{bufetov2019number}
A. I. Bufetov, P. P. Nikitin, Y. Qiu.
On number rigidity for Pfaffian point processes.
Mosc. Math. J. 19 (2019), no. 2, 217-274.

%已用
\bibitem{Bufetov2021}
A. I. Bufetov, F. D. Cunden, Y. Qiu.
Conditional measures for Pfaffian point processes: conditioning on a bounded domain.
Ann. Inst. Henri Poincar\'e Probab. Stat. 57 (2021), no. 2, 856-871.


%\bibitem{Costin 1995}
%O. Costin and J. Lebowitz.
%Gaussian fluctuations in random matrices,
%Phys. Rev. Lett. 75 (1): 69-72 (1995).
\bibitem{Costin1995}
O. Costin and J. Lebowitz.
Gaussian fluctuations in random matrices.
Phys. Rev. Lett. 75 (1995), no. 1, 69-72.

%已用
\bibitem{Daley1988}
D. J. Daley, D. Vere-Jones.
An introduction to the theory of point processes. Vol. II.
General theory and structure. Second edition. Probability and its Applications (New York). Springer, New York, 2008. xviii+573 pp.


%被Leble引用的
\bibitem{DHLM-DLRequation}
D. Dereudre, A. Hardy, T. Leblé, M. Maïda.
DLR equations and rigidity for the sine-beta process.
Comm. Pure Appl. Math. 74 (2021), no. 1, 172-222.

%这个有关CLT提一下
\bibitem{Deleporte&Lambert2020}
A. Deleporte, G. Lambert, (2024). Widom's conjecture: Variance asymptotics and entropy bounds for counting statistics of free fermions. ArXiv. https://arxiv.org/abs/2405.07796
%Free femions
%A. Deleporte and G. Lambert. Central limit theorem for smooth statistics of one-dimensional free fermions. (arXiv:2304.12275), May 2023.
% A. Deleporte and G. Lambert. Universality for free fermions and the local Weyl law for semiclassical Schrödinger operators. J. Euro. Math. Soc., 2024.


\bibitem{Du}
R. Durrett.
Probability: theory and examples.
Fifth edition. Cambridge Series in Statistical and Probabilistic Mathematics, 49. Cambridge University Press, Cambridge, 2019.

%已用
\bibitem{DysonFreemanJ.1970}
F. J. Dyson.
Correlations between eigenvalues of a random matrix.
Comm. Math. Phys. 19 (1970), 235-250.


%Sine4
\bibitem{Forrester2010}
P. J. Forrester.
Log-gases and random matrices.
London Mathematical Society Monographs Series, 34. Princeton University Press, Princeton, NJ, 2010. xiv+791 pp.

%用在Pfaffian过程的介绍里
\bibitem{forrester2011pfaffian}
P. J. Forrester, A. Mays.
Pfaffian point process for the Gaussian real generalised eigenvalue problem.
Probab. Theory Related Fields 154 (2012), no. 1-2, 1-47.

%被Leble引用
\bibitem{Johansson1998}
K. Johansson.
On fluctuations of eigenvalues of random Hermitian matrices.
Duke Math. J. 91 (1998), no. 1, 151-204.


%已用
%\bibitem{kallenberg1986random}
%O. Kallenberg.
%Random measures, theory and applications.
%Probability Theory and Stochastic Modelling, 77. Springer, Cham, 2017. xiii+694 pp.
\bibitem{kallenberg1986random}
O. Kallenberg.
Random measures, theory and applications.
Probab. Theory Stoch. Model. 77. Springer, Cham, 2017, xiii+694 pp.


%已用
\bibitem{Kargin_2013}
V. Kargin.
On Pfaffian random point fields.
J. Stat. Phys. 154 (2014), no. 3, 681-704.


%用到pfaffian过程的定义里--Pf
\bibitem{Koshida_2021}
S. Koshida.
Pfaffian point processes from free fermion algebras: perfectness and conditional measures.
SIGMA Symmetry Integrability Geom. Methods Appl. 17 (2021), Paper No. 008, 35 pp.

%这个好像没用到
%\bibitem{Lebowitz1963}
%J. L. Lebowitz, J. K. Percus.
%Statistical thermodynamics of nonuniform fluids.
%J. Mathematical Phys. 4 (1963), 116-123.

\bibitem{Lambert2021}
G. Lambert.
Mesoscopic central limit theorem for the circular $\beta$-ensembles and applications.
Electron. J. Probab. 26 (2021), Paper No. 7, 33 pp.

%sine-beta
\bibitem{Leble2021}
T. Leblé.
CLT for fluctuations of linear statistics in the sine-beta process.
Int. Math. Res. Not. IMRN 2021, no. 8, 5676-5756.

%已用
\bibitem{Lenard1973}
A. Lenard.
Correlation functions and the uniqueness of the state in classical statistical mechanics.
Comm. Math. Phys. 30 (1973), 35-44.

%已用
\bibitem{Lenard1975.1}
A. Lenard.
States of classical statistical mechanical systems of infinitely many particles. I.
Arch. Rational Mech. Anal. 59 (1975), no. 3, 219-239.

%已用
\bibitem{Lenard1975}
A. Lenard.
States of classical statistical mechanical systems of infinitely many particles. II. Characterization of correlation measures.
Arch. Rational Mech. Anal. 59 (1975), no. 3, 241-256.


\bibitem{Lu} E. Lukacs.
Characteristic functions.
Second edition, revised and enlarged. Hafner Publishing Co., New York, 1970.


%用在pfaffian过程介绍里
\bibitem{Matsumoto_2013}
S. Matsumoto, T. Shirai.
Correlation functions for zeros of a Gaussian power series and Pfaffians.
Electron. J. Probab. 18 (2013), no. 49, 18 pp.


%也没用到
%\bibitem{Mehta-book-RM}
%M. Mehta.
%Random matrices.
%Third edition
%Pure Appl. Math. (Amst.), 142
%Elsevier/Academic Press, Amsterdam, 2004. xviii+688 pp.
%ISBN:0-12-088409-7



%%DPP相关，没用到
%\bibitem{Olshanski_2020}
%G. Olshanski,
%Determinantal point processes and fermion quasifree states.
%Comm. Math. Phys. 378 (2020), no. 1, 507-555.

%可以用到Pfaffian里
\bibitem{rains2000correlation}
E. M. Rains.
Correlation functions for symmetrized increasing subsequences.
arXiv: math/0006097, 2000.

%这个本来是想用Ursell函数的，现在也用不到
%\bibitem{Soshnikov1998-LSDforGF}
%A. Soshnikov.
%Level spacings distribution for large random matrices: Gaussian fluctuations.
%Ann. of Math. (2) 148 (1998), no. 2, 573-617.

%已用
\bibitem{Soshnikov2000fluctuations}
A. B. Soshnikov.
Gaussian fluctuation for the number of particles in Airy, Bessel, sine, and other determinantal random point fields.
J. Statist. Phys. 100 (2000), no. 3-4, 491-522.

%已用
\bibitem{Soshnikov2000CLT}
A. Soshnikov.
The central limit theorem for local linear statistics in classical compact groups and related combinatorial identities.
Ann. Probab. 28 (2000), no. 3, 1353-1370.

%已用
\bibitem{Soshnikov2000}
A. Soshnikov.
Determinantal random point fields.
Uspekhi Mat. Nauk 55 (2000), no. 5 (335), 107-160; translation in
Russian Math. Surveys 55 (2000), no. 5, 923-975.

%已用
\bibitem{Soshnikov2002}
A. Soshnikov.
Gaussian limit for determinantal random point fields.
Ann. Probab. 30 (2002), 171-187.

%已用
\bibitem{soshnikov2003janossy}
A. Soshnikov.
Janossy densities. II. Pfaffian ensembles.
J. Statist. Phys. 113 (2003), no. 3-4, 611-622.

%介绍Pf的性质，可以不用
%\bibitem{deBruijn1955}
%N. G. de Bruijn,
%On some multiple integrals involving determinants.
%J. Indian Math. Soc. (N.S.) 19 (1955), 133-151 (1956).


%新添加的一些

%用过了
%\bibitem{Soshnikov2000Gf-for-DPP}
%A. B. Soshnikov,
%Gaussian fluctuation for the number of particles in Airy, Bessel, sine, and other determinantal random point fields.
%J. Statist. Phys. 100 (2000), no. 3-4, 491–522.

%\bibitem{Koshida2021}
%S. Koshida.
%Pfaffian Point Processes from Free Fermion Algebras: Perfectness and Conditional Measures.

%这个与sine过程有关
\bibitem{Valkó2020}
B. Valkó, B. Virág.
Operator limit of the circular beta ensemble.
Ann. Probab. 48 (2020), no. 3, 1286-1316.




%CUE
%\bibitem{citekey}
%Maples, Kenneth(CH-ZRCH); Najnudel, Joseph(1-CINC); Nikeghbali, Ashkan(CH-ZRCH)
%Strong convergence of eigenangles and eigenvectors for the circular unitary ensemble.(English summary)
%Ann. Probab. 47 (2019), no. 4, 2417–2458.







%\bibitem{}
%B. Valkó, B. Virág.
%Continuum limits of random matrices and the Brownian carousel.
%Invent. Math. 177 (2009), no. 3, 463-508.

\end{thebibliography}
\end{document}